\documentclass[10pt,reqno]{amsart}
\usepackage{amssymb,mathrsfs,color, pdfpages,slashed,pdfsync}
\usepackage{enumerate}
\usepackage{graphicx}
\usepackage{xcolor} 
\usepackage{geometry}
\usepackage{amsfonts,amssymb,amsmath}
\usepackage{slashed}
\usepackage{ mathrsfs }
\usepackage[titletoc,title]{appendix}
\usepackage{wrapfig}

\usepackage{cancel}

\usepackage{cite}

\usepackage{nth}

\usepackage{hyperref}

\usepackage{marginnote}

\definecolor{green}{rgb}{0,0.8,0} 

\definecolor{deepgreen}{cmyk}{0,4,0,0}

\newcommand{\Del}[1]{}

\numberwithin{equation}{section}

\newtheorem{theorem}{Theorem}[section]
\newtheorem{corollary}[theorem]{Corollary}
\newtheorem{lemma}[theorem]{Lemma}
\newtheorem{proposition}[theorem]{Proposition}

\newcommand{\pa}{\triangleright}

\usepackage{pgfplots}

\def\l{\langle}
\def\r{\rangle}

\def\f12{\frac 1 2}

\def\a{\alpha}

\def\ga{\gamma}

\def\De{\Delta}
\def\ep{\epsilon}
\def\la{\lambda}
\def\La{\Lambda}

\def\om{\omega}

\def\pa{\partial}
\def\les{\lesssim}

\def\e{\text{e}}

\def\th{\theta}
\def\Th{\Theta}
\newcommand{\D}{\mbox{$D \mkern-13mu /$\,}}

\begin{document}
	\title{On the motion of  charged particles in constant electromagnetic field: the parallel case}
	
	\author{Shuang Miao \and Shiwu Yang \and Pin Yu}

\date{}

\maketitle

\begin{abstract}
	This paper is devoted to  presenting a rigorous mathematical derivation for the classical phenomenon in Maxwell's  theory that  a charged particle moves along a straight line in  a constant electromagnetic field  if the initial velocity is parallel to the constant electromagnetic field. The particle is modeled by scaled solitons to a class of nonlinear Klein-Gordon equations and the nonlinear interaction between the charged particle and the electromagnetic field  is governed by the Maxwell-Klein-Gordon system. We show that when the size and amplitude of the particle are sufficiently small, the solution to the coupled nonlinear system exists up to any given time and the energy of the particle concentrates along a straight line. The method relies on the modulation approach for the study of stability for solitons and weighted energy estimates for the Maxwell-Klein-Gordon equations. 
\end{abstract}
\section{Introduction}

In the classical electromagnetic theory, the motion of a charged particle  generates electricity, which will then change the distribution of the electromagnetic field. In particular the interaction between the motion of the particle and the electromagnetic field is nonlinear. If the total charge of the particle is small compared to the background electromagnetic field so that their interaction can be neglected, then the motion of a charged particle in a constant electromagnetic field can be described in a simple way, depending on the initial velocity of the charged particle. If initially the charged particle moves along the direction of the electromagnetic field, then the particle moves at a constant speed along a straight line. If initially the particle moves in a direction which is perpendicular to the direction of the electromagnetic field, then the trajectory of the particle is a circle with radius depending on the total charge, the mass and the speed of the particle as well as the strength of the electromagnetic field. The motion of the particle for the other directions can be decomposed into the sum of these two basic scenarios.

Motivated by the related  works \cite{Hintz24:BH:geo:1},\cite{Hintz24:BH:geo:2}, \cite{Hintz24:BH:geo:3}, \cite{psedoStuart},  \cite{einsteinStuart}, \cite{yang4} on Einstein's geodesic hypothesis, the aim of the present paper is to give a rigorous mathematical verification of the above physical phenomenon for the case when the particle moves along the direction of the constant electromagnetic field. The test particle is modeled by a scaled complex scalar field governed by a class of nonlinear Klein-Gordon equation with small size and amplitude. Taking into account of the effect of the charged particle, the motion of a charged particle in electromagnetic field can be described by the nonlinear Maxwell-Klein-Gordon system (MKG). We assume that the scaled scalar field is close to some stable soliton and the background electromagnetic field is constant. For the case when the initial velocity of the particle is parallel to the direction of the background electromagnetic field, we show that the MKG system admits a solution to any given time $T>0$ with the property that the energy of the particle concentrates along a straight line parallel to the constant electromagnetic field and the Maxwell field is close to the background constant electromagnetic field. This in particular improves the previous short time result \cite{Stuart09:MKG} on background electromagnetic field with uniformly bounded connection field by Long-Stuart. 

\subsection{The Maxwell-Klein-Gordon system and stable solitons}
The nonlinear interaction between a massive  charged particle and the background electromagnetic field can be  described by the massive MKG equations. Let $A$ be a real valued $1$-form in Minkowski space $\mathbb{R}^{1+3}$. The covariant derivative $D_\mu$ acting on a complex valued scalar field $\phi$ via the relation 
\begin{align*}
	D_\mu \phi =\partial_\mu \phi+\sqrt{-1} A_\mu \phi. 
\end{align*}
The commutator of this covariant derivative gives the electromagnetic field $F$ with components
\begin{align}\label{connection def}
	F_{\mu\nu}=\partial_\mu A_\nu-\partial_\nu A_\mu,
\end{align}
which can also be viewed as the $2$-form $F=dA$. The MKG equation is  a system for the connection field $A$ and the scalar field $\phi_{\ep, \delta}$
\begin{equation}
	\label{Eq:MKGe-Fphi}
	\begin{cases}
		\partial^\mu F_{\mu\nu}=-J[\phi_{\ep, \delta}]_{\nu}=-\Im(\phi_{\ep, \delta}\cdot\overline{D_\nu \phi_{\ep, \delta}}),\\
		\Box_A\phi_{\ep, \delta}:=D^\mu D_\mu \phi_{\ep, \delta}=\mathcal{V}_{\ep, \delta}'(\phi_{\ep, \delta}).
	\end{cases}
\end{equation}
Here and throughout the paper the indices are raised and lowered with respect to the flat Minkowski metric. $\mathcal{V}_{\ep,\delta}(\cdot)$ is a potential function which will be specified later. We expect that the electromagnetic field $F$ is close to some constant field $F_b$. After Lorentz transformation, we may assume that the non-vanishing component of the constant electromagnetic field $F_b$ is $(F_b)_{12}=1$, that is the associated electric field vanishes and the magnetic field is $(0, 0, 1)$. The particle is approximated by a soliton in the following sense 
\begin{align}\label{phi scale}
	\phi_{\ep, \delta}=\delta \phi(\ep^{-1}t, \ep^{-1}x)
\end{align}
with some positive constants $\ep, \delta>0$, where $\ep$ stands for the size of the particle and $\delta$ measures the amplitude of the particle. We expect that the scaled scalar field $\phi$ is close to a stable soliton to the nonlinear Klein-Gordon equation
\begin{equation}
	\label{eq:NLW:p}
	\Box \phi-m^2 \phi+|\phi|^{p-1}\phi=0
\end{equation}
in the Minkowski space $\mathbb{R}^{1+3}$, where $m>0$ is a constant, denoting the mass of the particle  and 
$$\Box:=-\partial_{t}^{2}+\sum_{i=1}^{3}\partial_{i}^{2}.$$
The solitons or stationary waves are the special solutions of the form $\e^{i\om t}f_\om(x)$ such that 
\begin{equation}
	\label{eq:ground}
	\Delta f_\om-(m^2-\om^2)f_\om+|f_\om|^{p-1}f_\om=0.
\end{equation}
The Lorentzian symmetries of the equation and the Minkowski spacetimes can generate a family of special solutions to the full equation \eqref{eq:NLW:p}. Denote 
\begin{equation*}
	\la=(\om, \theta,\xi, u)\in \La  \equiv\left\{(\om, \theta, \xi, u)\in \mathbb{R}^8: |u| <1,\quad \frac{p-1}{6-2p}<  \frac{\om^2}{m^2} < 1 \right\}.
\end{equation*}
In particular we may require that $ 1 <p <\frac{7}{3}$. This set corresponds to the stable solitons, for which the associated energy is convex in terms of the parameter $\omega$. For $\la\in \La$, define
\begin{align}
	\label{z}
	& z(x;\la)= \rho P_u(x-\xi)+(I-P_u)(x-\xi),\quad \rho=(1-|u|^2)^{-\f12},\\
	\notag
	& \Th(x;\la)=\th-\om u\cdot z(x;\la),\quad P_u y=\frac{y\cdot u}{|u|^2}u.
\end{align}
Here $P_u$ is the projection operator in the direction $u\in \mathbb{R}^3$ and $I$ is the identity map. Let
\begin{equation}\label{def: phiS, psiS}
	\begin{split}
		&\phi_S(x;\la):=\e^{i\Th(x;\la)}f_\om(z(x;\la)),\\
		& \psi_S(x;\la):=\e^{i\Th(x;\la)}\left(i\rho\om
		f_{\om}(z(x;\la)) -\rho u\cdot \nabla_z f_\om(z(x;\la))\right).
	\end{split}
\end{equation}
Rewrite the system \eqref{Eq:MKGe-Fphi} in terms of $\phi$ instead of $\phi_{\ep, \delta}$. Let 
\begin{equation}\label{connection scale}
	A^{\ep}(t, x)=\ep A(\ep t, \ep x),\quad F^\ep (t, x)=\ep^2 F(\ep t, \ep x). 
\end{equation}
Then we can compute that 
\begin{align*}
	&\Box_{A}\phi_{\ep, \delta}(t, x)=\delta \ep^{-2} \Box_{A^\ep}\phi (\ep^{-1}t, \ep^{-1}x ), \\
	& \pa^\mu F_{\mu\nu}(t, x)= \ep^{-3} \pa^\mu F^\ep_{\mu\nu}(\ep^{-1}t, \ep^{-1}x)=-\delta^2 \ep^{-1}\Im(\phi\cdot \overline{D_\nu   \phi }).
\end{align*}
This means that we are reduced to study the rescaled system  
\begin{equation}
	\label{eq:MKG:scaled}
	\begin{cases}
		\partial^\mu F_{\mu\nu}^\ep =-\delta^2 \ep^2 J[\phi]_{\nu}=-\delta^2 \ep^2 \Im(\phi \cdot\overline{D_\nu  \phi}),\\
		\Box_{A^\ep }\phi=\mathcal{V}'(\phi)
	\end{cases}
\end{equation}
describing the motion of a large massive charged particle on a slowly varying electromagnetic field. 
For simplicity  we consider the pure power case and take the potential $\mathcal{V}$ to be
\begin{align}
\label{eq:scaledV}
	\mathcal{V}(\phi)=\frac{m^2}{2}|\phi|^2-\frac{1}{p+1}|\phi|^{p+1}=\delta^{-2}\ep \mathcal{V}_{\ep, \delta}(\phi_{\ep, \delta}).
\end{align}
Consider the Cauchy problem to the MKG system \eqref{eq:MKG:scaled} with initial data $(E^\ep(0), B^\ep(0), \phi_0, \phi_1)$ such that the following compatibility condition hold
\begin{align}\label{EM initial constraint}
	\begin{split}
		\phi(0, x)=\phi_0(x),\quad D_t\phi(0, x)=\phi_1, \\ 
		\textrm{div} (E^\ep(0))=-\delta^2\ep^2 \Im(\phi_0 \cdot\overline{   \phi_1}),\quad \textrm{div}(B^\ep(0))=0. 
	\end{split}
\end{align}
Here $(E,B)$ is the canonical decomposition of the electromagnetic field:
\begin{align}\label{EM decom}
	F_{i0}:=E_{i},\quad F_{ij}:=-\epsilon_{ijk}B^{k},\quad   1\leq i,j,k\leq 3,
\end{align}
in which  $\epsilon_{ijk}$ is the canonical volume form of $\mathbb{R}^{3}$. We remark here that $(E^\ep, B^\ep)$ scales  through the relation \eqref{connection scale}.

  Assume that the initial scaled  scalar field $\phi$ is close to some stable soliton up to second order derivatives in the following sense 
\begin{align}\label{soliton initial}
	\int_{\mathbb{R}^3} \left|D^{k+1}_j(\phi_0-\phi_S(x;\la_0))\right|^2 +\left|D^k_j(\phi_1-\psi_S(x;\la_0))\right|^2dx \leq \ep^2,\quad j\leq 3, k\leq 2
\end{align}
for some $\la_0\in \La$. 
In addition, we assume that the following weighted energy bound  
\begin{align}\label{soliton initial weight}
	\int_{\mathbb{R}^3} (1+|x|^2)\left(\left|D_j(\phi_0-\phi_S(x;\la_0))\right|^2 +\left|\phi_0-\phi_S(x;\la_0)\right|^2+\left| \phi_1-\psi_S(x;\la_0)\right|^2 \right)dx \leq \ep^2,\quad j \leq 3.
\end{align}
Since we are assuming that the initial electromagnetic field is close to the constant constant electromagnetic field $(E=0, B=(0, 0, 1))$, after scaling, we could assume that the initial electric field $E^\ep(0)$ is small and the initial magnetic field $B^{\ep}(0)$ is close to the constant field $\ep^2(0, 0, 1)$. More precisely, we assume that 
\begin{align}\label{EM initial}
	\int_{\mathbb{R}^3} |\nabla^k E^\ep(0)|^2+|\nabla^k(B^\ep(0)-\ep^2(0, 0, 1))|^2 dx\leq \delta^2\ep^2,\quad k\leq 2. 
\end{align}
Here $\delta$ is a small constant which is independent of $\ep$. We will derive the main estimates under Lorentz gauge condition:
\begin{align}\label{Lorentz gauge}
	\partial^{\mu}A_{\mu}^{\epsilon}=0,
\end{align}
We also consider the difference of the connection:
\begin{align}\label{connection diff}
	\tilde{A}^{\epsilon}:=A^{\epsilon}-A^{\epsilon}_{b}
\end{align}
with the explicit connection 1-form $A^{\epsilon}_{b}$:
\begin{align}\label{connection background}
	(A_{b}^{\epsilon})_{0}=(A^{\epsilon}_{b})_{3}=0,\quad (A^{\epsilon}_{b})_{1}=-\frac12x_{2}\epsilon^{2},\quad (A_{b}^{\epsilon})_{2}=\frac12x_{1}\epsilon^{2}.
\end{align}
Now we are now able to state our main theorem. 
\begin{theorem}\label{main thm}
Consider the Cauchy problem to the system \eqref{eq:MKG:scaled}.
	Let the assumptions \eqref{EM initial constraint},  \eqref{soliton initial}, \eqref{soliton initial weight}, and \eqref{EM initial} hold. Let $T>0$ be any given time. Then for all $2\leq p < \frac{7}{3}$ and all $\la_0=(\om_0, \th_0, 0, u_0)\in\La $ such that $u_0$ is parallel to $B_b=(0,0, 1)$, there exists a positive number $\ep^*$ depending only on  $\la_0$, $m$, $p$ and $T$ such that for all positive constants  $\delta, \ep <\ep^*$, 
	there exists a unique solution $(F^{\epsilon}(t,x), \phi(t, x))$ defined on $  [0, T/\ep]\times  \mathbb{R}^3$ to the equation \eqref{eq:MKG:scaled} with the
	following property: there is a $C^1$ curve $$\la(t)=(\om(t), \th(t),	\eta(t)+u_0t, u(t)+u_0)\in\La $$ such that
	\begin{equation*}
		|\la(0)-\la_0|\leq C\ep,\quad
		| (\dot\om(t), \dot\th(t), \dot\eta(t), \dot u(t))|\leq C\ep^2,\quad \forall
		t\in[0,T/\ep]
	\end{equation*}
	and the solution $\phi$ is close to the translated solitons
	\begin{equation*}
		\begin{split}
			\|\phi(t, x)-\phi_S(x;\la(t))\|_{H^3(\mathbb{R}^3)}+\|\pa_t\phi(t, x)-\psi_S(x;\la(t))\|_{H^2(\mathbb{R}^3)}
			\leq C\ep,\quad \forall t\in [0, T/\ep]
		\end{split}
	\end{equation*}
	under the Lorentz gauge condition. Moreover the electromagnetic field (or the connection field) is close to the constant electromagnetic field in the following sense
	\begin{align*}
		\|\tilde{A}^{\epsilon}(t,\cdot)\|_{L^{\infty}}^{2}+\|\partial\partial^{s}\tilde{A}^{\epsilon}(t,\cdot)\|_{L^{2}}^{2}\leq C \epsilon^{2},\quad s\leq 2,\quad \forall t\leq T/\ep.
	\end{align*}
	Here  the constant $C$ depends only on $m$, $p$, $T$ and $\la_0$.
\end{theorem}
Back to the original problem regarding the motion of charged particles on constant electromagnetic field or the nonlinear  system \eqref{Eq:MKGe-Fphi}, assume that the potential $\mathcal{V}_{\ep, \delta }$ verifies  the scaling relation   \eqref{eq:scaledV} and the scalar field $\phi_{\ep, \delta}$  is close to some scaled stable soliton traveling in the direction of  the constant magnetic field. If the initial electromagnetic field is close to the constant electromagnetic field $(E=0, B=(0, 0, 1))$, then as long as the size $\ep$  and   amplitude $\delta$ of the particle  are sufficiently small the particle stays close to a straight line parallel to the direction of the magnetic field up to any given time $T>0$.
In particular for the case when the charged particle moves along the direction of the constant magnetic field, the trajectory of the particle is a straight line up to any given large time. This gives a rigorous mathematical verification for the classical phenomenon in electromagnetic theory for the case when the  initial direction of the  particle is parallel to the constant magnetic field. 

We emphasize here that the small constant $\delta$ does not rely on the small parameter $\ep$. In other words, the smallness of the  amplitude of the particle is independent of the size of the particle. This is one of the major improvements compared to the previous results in \cite{Eamonn06:MKG},
\cite{Stuart09:MKG},  in which the above stability result was established under the assumption $\delta<\ep^{\frac{1}{2}}$. Moreover, our result indicates that the particle moves along  a straight line up to any given time $T>0$ while the stability result was only shown for a short time in \cite{Stuart09:MKG}. Most importantly, the background electromagnetic field in Long-Stuart's work was required to decay in the sense that the associated connection field $A$ (see the relation \eqref{connection def} between the electromagnetic field and the connection field) is uniformly bounded, which excludes the most interesting case  that the background electromagnetic field is constant.

\subsection{Related works}

Existence of solitons to the MKG system with small coupling constant had been studied in \cite{Eamonn06:MKG}. Such soliton could be viewed as small perturbation of the solution $(A=0, \phi=e^{i\omega t}f_{\omega}(x))$ with vanishing connection field. Orbital stability on a fixed slowly varying electromagnetic background had also been proven in \cite{Eamonn06:MKG}. Short time stability for the coupled MKG system had then been established in \cite{Stuart09:MKG} on a background electromagnetic field with uniformly bounded connection field. The analogue of Einstein's geodesic hypothesis which says that a massive test particle moves along a timelike geodesic in spacetime was studied in \cite{psedoStuart}, \cite{einsteinStuart}, \cite{yang4}. In this situation the motion of the particle is governed by the gravitation of the spacetime. Gluing black hole spacetime along a timelike geodesic  is recently shown to be possible in a series works \cite{Hintz24:BH:geo:1}, \cite{Hintz24:BH:geo:2}, \cite{Hintz24:BH:geo:3} by Hintz.

\subsection{Comments on the proof}
We now briefly elaborate the main ideas for the proof. We rely on modulation approach to study stability of solitons and weighted  energy estimates for solutions to the MKG system to control the error terms arising from the corresponding growing connection field at spatial infinity.

Our improvement on the stability time (from short time to arbitrary large given time) and the size of the amplitude (the small constant $\delta$ is independent of the size of the particle $\epsilon$) is inspired by the work \cite{yang4} of the second author on Einstein's geodesic hypothesis. 
To solve the nonlinear system \eqref{eq:MKG:scaled}, we use bootstrap argument. We first investigate the orbital stability of solitons on a fixed slowly varying electromagnetic field which is close to the constant electromagnetic field $(E=0, B=(0, 0, \ep^2))$. Under the bootstrap assumption 
\begin{equation}
\label{eq:bt:intro}
\|\pa \pa^s \tilde{A}^\ep(t, x) \|_{L_x^2}\leq 10^3 \ep^2,\quad s\leq 2, \quad t\leq T/\ep,
\end{equation}
in which $\tilde{A}^\ep $ is the difference of the full connection field $A^\ep$ and the connection field associated to the given constant electromagnetic field, we establish the orbital stability of solitons up to time $T/\ep$ for any given $T>0$. 
In Minkowski spacetime, the orbital stability for all time follows by decomposing the conserved energies and mass around the soliton under the othogonality condition. The associated energies and mass are no longer conserved in the presence of connection field. The stability time then relies on the growth of these almost conserved quantities. The key observation is that the leading error terms (see Proposition \ref{propcon}) 
\[
\int_0^t \int_{\mathbb{R}^3}    \langle   i X(A^\ep)^\mu \pa_\mu\phi_S , \phi_S\rangle dx ds
\] 
is of size $\ep^2$ instead of $\ep$ for all time $t\leq T/\ep$ after integration by parts and the fact that 
\[
|X \phi_S|\les \ep,\quad X=\pa_t +u_0^k\pa_k.
\]
Here $u_0$ is the initial moving direction of the particle which by  our assumption is also the direction of the constant magnetic field. The above smallness follows from the fact that the soliton (or the particle) moves along   the constant magnetic field. This is one of the key points allowing us to obtain orbital stability result to time $T/\ep$.

Another issue arises from the growing connection field associated to the given constant electromagnetic field. A typical such error term is of the form 
\begin{align*}
\int_{\mathbb{R}^3} |A_b^\ep|^2 |v|^2 dx=\int_{\mathbb{R}^3} \ep^2 |x|^2 |v|^2dx.
\end{align*}
Here $v\approx \phi -\phi_S$ is the remaining error term of the full solution and the soliton.  This term is a small error term if the connection field $A_b$ is uniformly bounded  as required in Long-Stuart's work \cite{Stuart09:MKG}. However such restriction does not include the case when the background electromagnetic field is constant. The idea to control the above error term arising from the growing background connection field, which is inspired by the works \cite{YangYu:MKG:smooth}, \cite{Yang:mMKG} on the study of global dynamics of MKG system, is to derive weighted energy estimate for the full solution out side a forward light cone. 

Notice that the soliton $\phi_S$ decays exponentially. It suffices to  derive the following type of weighted energy estimates for the full solution
 \begin{align*}
	\int_{t+R_0\leq r } (1+|x|^2) (|D\phi|^2+m^2|\phi|^2)dx \les  1. 
\end{align*}
Here $R_0$ is chosen such that outside the ball with radius $R_0$ on the initial hypersurface the initial data are small 
\begin{align*}
	\int_{|x|\geq R_0} (1+|x|^2)(|\nabla\phi_0|^2+|\phi_0|^2+|\phi_1|^2) dx\leq 10\ep^2.
\end{align*}
In view of finite speed of propagation for nonlinear wave equations, we could first understand the full solution in the exterior region $\{t+R_0\leq r\}$ with small initial data. Since it suffices to derive a weighted energy estimate for the solution, we rely on the energy method for massive MKG equations used in \cite{Yang:mMKG}. See more details in Section \ref{sec:weightE}. 

Once we have stability results for solitons on fixed background electromagnetic field verifying the above bootstrap assumption \eqref{eq:bt:intro}, to close the argument and improve these bootstrap assumption, we analyze the Maxwell equation   under the Lorentz gauge condition (first equation of \eqref{eq:MKG:scaled}), which is equivalent to the wave equation 
\begin{align*}
	\pa^\mu F^\ep_{\mu\nu}=\Box A^{\ep}_{\nu}-\pa_\nu\pa^\mu A^\ep_\mu =\Box \tilde{A}^\ep =-\delta^2\ep^2 J[\phi]_{\nu}
\end{align*}
for the connection field $\tilde{A}^\ep$. Since the full solution $\phi$ is close to the soliton which is of size $1$, the standard energy estimate roughly leads to the bound 
\begin{align*}
\|\pa \tilde{A}^\ep\|_{L^2}\les \ep^2 +\delta^2 \ep^2\int_0^t \|J[\phi]\|_{L^2}ds\les \ep^2 +\delta^2 \ep^2 t\leq \ep^2 +T\delta^2 \ep.
\end{align*}
If we want to improve the bootstrap assumption, we may need to require that 
\begin{align*}
\delta^2\leq \ep,\quad \delta \leq \ep^{\frac{1}{2}}. 
\end{align*}
The key observation to improve the size of the amplitude $\delta$ is that the above mentioned vector field $X=\pa_t +u_0^k\pa_k$ is timelike and one can use this timelike vector field as multiplier instead of $\pa_t$ to do the energy estimate. Notice that the full solution $\phi$ is close to the soliton $\phi_S$,  the main part of $J[\phi]$ is $J[\phi_S]$. Integration by parts then indicates that 
\begin{align*}
\int_0^t \int_{\mathbb{R}^3} J[\phi_S] X\phi_S dxdt=O(1)- \int_0^t \int_{\mathbb{R}^3} XJ[\phi_S] \phi_S dxdt.
\end{align*} 
As we have pointed out that the soliton travels along the direction $X$ in spacetime, we could gain an extra smallness 
\[
|XJ[\phi_S]|\les \ep. 
\]
This allows us to improve the bootstrap assumption for the connection field $\tilde{A}^\ep$ under the condition that the amplitude $\delta$ is sufficiently small which is independent of the size $\ep$ of the particle.

The paper is organized as follows: In Section \ref{review soli Mink} we  review some properties of stable solitons to nonlinear Klein-Gordon equation in Minkowski space and briefly introduce the modulation approach. In Section \ref{stability soli fixed EM}, we show that the stable solitons are orbital stable up to time $T/\ep$  in a fixed electromagnetic field which is close to the constant magnetic field $B^{\ep}=(0,0,\ep^{2})$. We also establish the higher order energy estimates and weighted energy estimates for the full solution. Finally in Section \ref{recover EM} we derive energy estimates for the connection field and improve the bootstrap assumption for the connection field used in Section \ref{stability soli fixed EM}, hence finishing the proof for the main Theorem \ref{main thm}.

\noindent \textbf{Acknowledgments.} S. Miao is supported by the National Key R\&D Program of China 2021YFA1001700 and the National Science Foundation of China 12426203, 12221001. S. Yang is supported by the National Key R\&D Program of China 2021YFA1001700 and  the National Science Foundation of China  12171011,  12141102, 12426203. P. Yu is  supported by the National Science Foundation of China 11825103, 12141102 and MOST-2020YFA0713003.

\section{Stable solitons in Minkowski space} \label{review soli Mink}
In this section, we review the necessary ingredients for solitons to the nonlinear wave equation \eqref{eq:NLW:p}. 
It has been shown, for example in \cite{manysolLions}, that the  elliptic equation \eqref{eq:ground} admits infinite many solutions, of which the ground state minimizes the associated energy
\[
E_\om(v)=\f12\int_{\mathbb{R}^3}|\nabla
v|^2+(m^2-\om^2)|v|^2-\frac{2}{p+1}|v|^{p+1}dx
\]
for fixed $\omega\in (-m, m)$. We summarize here the properties for the ground state which are crucial  to our subsequent analysis. 
\begin{proposition}
	\label{prop:ground:f}
	For $1<p< 5$ and  $-m<  \om  <m$, there exists a unique, positive, radial symmetric solution $f_\om(x)\in H^4(\mathbb{R}^3)\cap C^4(\mathbb{R}^3)$ to equation \eqref{eq:ground}, which is decreasing in $|x|$ with the following properties:
	\begin{itemize}
		\item[1]. Exponential decay up to fourth order derivatives
		\begin{equation*}
			|\nabla^\a f_{\om}(x)|\leq C(\om)e^{-c(\om)|x|},\quad \forall x\in \mathbb{R}^3, \quad |\a|\leq 4
		\end{equation*}
		for some positive constant $c(\om)$;
		\item[2]. Asymptotic behavior
		\begin{equation*}
			\lim\limits_{|x|\rightarrow\infty}\frac{f_{\om}'(|x|)}{f_{\om}(|x|)}=-\sqrt{m^2-\om^2};
		\end{equation*}
		\item[3]. Scaling of the solutions
		\begin{equation*}
			f_\om(x)=(m^2-\om^2)^{\frac{1}{p-1}}f(\sqrt{m^2-\om^2}x),
		\end{equation*}
		where $f(x)$ is the solution for $m^2-\om^2=1$;
		\item[4]. Energy identities 
		\begin{equation*}
			\frac{3(p-1)(m^2-\om^2)}{2\|\nabla f_\om\|_{L^2}^2}=\frac{5-p}{2\|f_\om\|_{L^2}^2}
			=\frac{(p+1)(m^2-\om^2)}{\|f_\om\|_{L^{p+1}}^{p+1}}.
		\end{equation*}
	\end{itemize}
\end{proposition}
Existence of ground state has been shown in \cite{exsolitonLions}, \cite{odelions}, \cite{exsoliStrauss}. K. McLeod \cite{uniqueSolMcLeod} proved the uniqueness of the ground state. Smoothness of the ground state as well as asymptotic behaviors could be found in \cite{exsolitonLions},  \cite{moduStuart}, \cite{decaySerrin}. Using integration by parts, the  energy identities follow by multiplying the equation \eqref{eq:ground} with $f_\om$, $x\cdot \nabla f_\om$ respectively.

For the curve $\la(t)=(\om(t), \th(t), \xi(t), u(t))$ such that 
\begin{equation*}
	\dot{\om}=0,\quad \dot{\th}=\frac{\om}{\rho}, \quad \dot{\xi}=u, \quad\dot{u}=0,
\end{equation*}
it can be shown by direct computation that $\phi_S(x;\la)$ defined in the introduction also solves the nonlinear wave equation \eqref{eq:NLW:p}. Here and in the following, we use the dot to denote the derivative with respect to time variable $t$.

For $\la\in \La$, denote
\begin{equation*}
	V(\la)=(0, \frac{\om}{\rho}, u, 0).
\end{equation*}
Then we have the crucial relation 
\begin{align*}
	\psi_S(\la;x)=\pa_\la \phi_S(\la;x)\cdot V(\la) 
\end{align*}
as inner product of vectors in $\mathbb{R}^8$. Moreover we obtain the important identity
\begin{equation}
	\label{idenofphiS}
	\Delta_x \phi_S-m^2\phi_S+|\phi_S|^{p-1}\phi_S-\pa_\la\psi_S \cdot V(\la)=0.
\end{equation}
Here the Laplacian operator $\Delta_x$ is with respect to the variable $x$ (similarly $\Delta_z$ is taken with respect to the variable $z$ defined in \eqref{z}). 

For a $C^1$ curve
$\la(t)=(\om(t), \theta(t), \xi(t), u(t))\in \La$, let $\ga(t)$ be the modulated curve such that
\begin{equation}
	\label{lagaV}
	\dot{\la}=\dot{\ga}+V(\la),\quad \la(0)=\ga(0).
\end{equation}
Stability problems related to the above solitons  to the equation \eqref{eq:NLW:p} in Minkowski space has been studied extensively in the past decades. Orbital stability, the solution to the full equation \eqref{eq:NLW:p} exists for all time and stays close to some translated soliton, was first shown by   J. Shatah in \cite{stableShatah} for radial symmetric initial data and was later discussed in a general framework in \cite{stableShatah1}, \cite{stableShatah2}. Their approach relies on the fact that the energy $E_\om(f_{\om})$ is strictly convex in $\om$ if initially is close to some stable soliton $\phi_S(x; \lambda_0)$ with $\lambda_0\in \Lambda$. This condition on $\la_0$ is sharp in the sense that the solitons are unstable if the energy $E_\om(f_{\om})$ is concave in $\om$, see for example \cite{unstableshatah}, \cite{unstable2shatah}. Alternatively, the modulation approach, pioneered by M. Weinstein in
\cite{ModulWeinstein}, additionally controls the modulation curve, see the work of D. Stuart in \cite{moduStuart}.  Regrading the asymptotic stability of solitons, which says that the solution asymptotically approaches some soliton, we refer to the recent advances in \cite{Chengong24:KG:stab:soliton} and references therein. 

We now briefly describe the modulation approach. Notice that equation \eqref{eq:NLW:p} locally has a unique solution $(\phi(t, x), \pa_t\phi(t, x))\in C^1([0, t^*); H^1\times L^2)$. Decompose the solution $\phi$ as follows
\begin{equation}
	\label{eq:decomp}
	\begin{split}
		&\phi(t, x)=\phi_S(x;\la(t))+\e^{i\Th(x;\la(t))}v(t, x),\\
		&\pa_t\phi(t, x)=\psi_S(x;\la(t))+\e^{i\Th(x;\la(t))}w(t, x)
	\end{split}
\end{equation}
for a $C^1$ curve $\la(t)\in \La $ such that the following orthogonality condition hold
\begin{equation}
	\label{orthcond}
	\l \e^{-i\Th} \pa_\la \phi_S, w\r_{dx}=\l\e^{-i\Th}\pa_\la \psi_S, v\r_{dx}, \quad \forall
	t\in\mathbb{R}.
\end{equation} 
Here in this paper for complex valued functions $a(x)$, $b(x)$, $\l a(x), b(x)\r_{dx}$ stands for  the inner product in the function space $L^2(\mathbb{R}^3)$
\[
\l a, b \r_{dx}=\f12\int_{\mathbb{R}^3}a\bar b+\bar a b \quad dx.
\]
In view of the nonlinear equation \eqref{eq:NLW:p}, the orthogonality condition   \eqref{orthcond} leads to a coupled system of ODE's for modulation curve $\la(t)$ and PDE for the remaining terms $(v, w)$. 

To obtain quantitative estimates for the modulation curve $\la(t)$ as well as the error term $(v, w)$, we collect the well studied properties for the linearized operator around the ground state. Let
\begin{equation}
	\label{L+}
	\begin{split}
		& L_+=-\De_z+(m^2-\om^2)-pf_{\om}^{p-1}(z),\\
		& L_{-}=-\De_z+(m^2-\om^2)-f_{\om}^{p-1}(z),
	\end{split}
\end{equation}
be the linear operators associated to the linearization of the equation \eqref{eq:NLW:p} around the soliton. These operators verify the following properties proven for example in \cite{ModulWeinstein}.
\begin{proposition}
\label{propL}
	For fixed $\omega$, the linear operators  $L_-$, $L_+$ are self-adjoint operators in $L^2(\mathbb{R}^3)$ with the property  
	\[
	\ker L_-=span\{f_\om\},\quad \ker L_+=span\{\nabla_{z^i}f_\om\}_{i=1}^{3}.
	\]
	Moreover $L_-$ is non-negative and the strictly negative eigenspace of $L_+$ is of dimension one.
\end{proposition}
Since the linearized operator has negative eigenvalue, the orthogonality condition \eqref{orthcond} plays the role that the radiation term $(v, w)$ is orthogonal to the negative eigenspace. Therefore the following associated energy for the remainder term 
\begin{equation}
	\label{defofE0}
	E(v, w, \la)=\|w+\rho u\cdot \nabla_z v-i\rho \om v\|_{L^2(dz)}^2+\langle v_1,
	L_{+}v_1\rangle_{dz}+\langle v_2, L_{-}v_2\rangle_{dz} 
\end{equation}
is equivalent to $\|v\|_{H^1}+\|w\|_{L^2}$ under the orthogonality condition \eqref{orthcond}.
\begin{proposition}
	\label{positen} 
	Assume $\la\in \La $ and $v, w$ satisfy the orthogonality condition \eqref{orthcond}. Then there is a positive constant $C(\om, u)$, depending continuously on $\om$ and $u$, such that
	\[
	C^{-1}(\om, u)(\|w\|_{L^2}^2+\|v\|_{H^1}^2)\leq E(v, w, \la)\leq
	C(\om,u)(\|w\|_{L^2}^2+\|v\|_{H^1}^2).
	\]
\end{proposition}
Our subsequent analysis relies on this proposition. The proof could be found for example in \cite{moduStuart}, based on the properties for the linear operator. Convexity of the conserved energies then leads to the  control for the radiation terms, which then leads to the  bound for the modulation curve. 

To avoid too many constants, throughout this paper, we use the notation  $B\les C$ to stands for $B\leq D C$ for some constant $D$ depending only on $\lambda_0\in \Lambda$, $p$, $m$ and the given time $T>0$. We emphasize here that the implicit constant does not rely on the small parameter $\epsilon$ and $\delta$.   

\section{Stability of stable solitons on a fixed electromagnetic field}\label{stability soli fixed EM}
Let $A^{\ep}, \tilde{A}^{\ep}, A^{\ep}_{b}$ be given as in \eqref{connection def}, \eqref{connection diff}, \eqref{connection background} respectively and all satisfy the Lorentz gauge condition \eqref{Lorentz gauge}. Consider the Cauchy problem to the nonlinear wave equation
\begin{equation}
	\label{eq:CSF:fix}
	\begin{cases}
		\Box_{A^\ep} \phi-m^2 \phi+|\phi|^{p-1}\phi=0,\\
		\phi(0, x)=\phi_0(x),\quad \pa_t\phi(0, x)=\phi_1(x).
	\end{cases}
\end{equation}
We assume that connection field  $\tilde{A}^\ep$ verifies the following bound
\begin{equation}
	\label{eq:cond4:A}
	\|\pa\pa^s\tilde{A}^\ep(t,x) \|_{L_x^2} \leq 10^3\ep^2,\quad s\leq 2,\quad \forall t\leq T/\ep
\end{equation}
and initially $\|\tilde{A}^\ep(0, x)\|_{L_x^\infty}$ is small, that is, 
\begin{equation}
\label{eq:A0:id}
\|\tilde{A}^\ep(0, x)\|_{L_x^\infty}\les \ep. 
\end{equation}
This assumption can be realized under the Lorentz gauge condition and the gauge invariant assumptions on the initial data, which will be discussed later. More importantly, it allows the existence of nonzero charge of the particle. Therefore the above assumption indicates that 
\begin{align*}
	|\tilde{A}^\ep(t, x) | \les |\tilde{A}^\ep(0, x)| +\int_0^t |\pa_t\tilde{A}^\ep(t', x)| dt' \les \ep+t \ep^2 \les \ep,\quad \forall t\leq T/\ep. 
\end{align*}
On the other hand, for fixed time $t$, using Hardy's inequality, we obtain that 
\[
\|\tilde{A}^\ep(t, x)/|x-\xi(t)|\|_{L_x^2}\les \|\nabla \tilde{A}^\ep(t, x)\|_{L^2}\les \ep^2, 
\]
which implies that 
\[
|\tilde{A}^\ep(t, x)|\les \ep^2,\quad \forall  |x-\xi(t)|\leq 1. 
\] 
Here $\xi(t)$ is the center of the soliton. Now combining with the standard Sobolev embedding  
\[
|\tilde{A}^\ep(t, x)-\tilde{A}^\ep(t, y)|\les \ep^2 |x-y|^{\frac{1}{2}},
\]
we can demonstrate that 
\begin{equation}
	\label{eq:tildeA:bd}
	|\tilde{A}^\ep(t, x)| \les  \ep \min\{  1,   \ep (1+|x-\xi(t)|)^{\frac{1}{2}}\},\quad \forall x\in \mathbb{R}^3,\quad t\leq T/\ep. 
\end{equation}
The  main purpose of this section is to prove the orbital stability of the soliton on a small constant electromagnetic field up to time $T/\ep$. 
\begin{proposition}
	\label{prop:fix}
	Let $A^{\ep}$,  $\tilde{A}^{\ep}$,   $A^{\ep}_{b}$ be  connection fields verifying the Lorentz gauge condition and the assumptions \eqref{connection diff}, \eqref{connection background}, \eqref{eq:cond4:A}, \eqref{eq:A0:id}. Assume that $2\leq p < \frac{7}{3}$. Then for all $\la_0=(\om_0, \th_0, 0,
	u_0)\in\La $ such that $u_0$ is parallel to the direction of the magnetic field  $B_b$, there exists a positive number $\ep^*$ depending on  $\la_0$ such that for all positive $\ep <\ep^*$, if the initial data $\phi_0$, $\phi_1$ are close to some stable soliton in the sense that 
	\begin{equation}
		\label{eq:ID:sca:w0}
		\int_{\mathbb{R}^3} (1+|x|^2) (|\phi_0(x)-\phi_S(x;\la_0) |^2 +|\nabla(\phi_0(x)-\phi_S(x;\la_0)) |^2+ |\phi_1(x)-\psi_S(x;\la_0)|^2 )dx \leq \ep^2,
	\end{equation}
	then there exists a unique solution $\phi(t, x)$ defined on $  [0, T/\ep]\times  \mathbb{R}^3$ to the equation \eqref{eq:CSF:fix}  with the
	following property: there is a $C^1$ curve $\la(t)=(\om(t), \th(t),
	\eta(t)+u_0t, u(t)+u_0)\in\La $ such that
	\begin{equation*}
		|\la(0)-\la_0|\leq C\ep,\quad
		|\dot \ga|=|\dot \la-V(\la)|\leq C\ep^2,\quad \forall
		t\in[0,T/\ep]
	\end{equation*}
	and the solution $\phi$ is close to the translated solitons
	\begin{equation*}
		\begin{split}
			\|\phi(t, x)-\phi_S(x;\la(t))\|_{H^1(\mathbb{R}^3)}+\|\pa_t\phi(t, x)-\psi_S(x;\la(t))\|_{L^2(\mathbb{R}^3)}
			\leq C\ep,\quad \forall t\in [0, T/\ep].
		\end{split}
	\end{equation*}
	Here  the constant $C$ depends only on $m$, $p$, $T$ and $\la_0$.
\end{proposition}
The above orbital stability result will be shown under the above assumption \eqref{eq:cond4:A} for the connection field $\tilde{A}^\ep$, which verifies the Maxwell equation \eqref{eq:MKG:scaled}. The assumption for $\tilde{A}^\ep$ can be viewed as bootstrap assumptions. To close this bootstrap assumption, we will do energy estimates for $\tilde{A}^\ep$, which requires higher order energy estimates for the charge and current density $J[\phi]$, that is higher order energy estimates for the scalar field $\phi$. 
\begin{proposition}
	\label{pro:phi:HSob}
	Assume that the connection field $ {A}^{\ep}$ satisfies the estimate \eqref{eq:cond4:A} and \eqref{eq:A0:id}.  In addition to the assumption \eqref{eq:ID:sca:w0} on the initial data $(\phi_0, \phi_1)$, assume that the higher covariant energy is also close to the soliton with parameter $\la_0$
	\begin{equation*}
		\sum\limits_{k\leq 2}\int_{\mathbb{R}^3} |D_j^{k+1}(\phi_0(x)-\phi_S(x;\la_0))|^2 +|D_j^k(\phi_1(x)-\psi_S(x;\la_0))|^2 dx \leq \ep^2,\quad j=1, 2, 3. 
	\end{equation*}
	Let $\tilde{\la}(t)$ be the integral curve of $V(\la)$ with $\tilde{\la}(0)=\la(0)$ and the modulation curve $\la(t)$ obtained in  the previous Proposition \ref{prop:fix}. Then
	\[
	\sum\limits_{|k|\leq
		3} \int_{\mathbb{R}^3}|D^k(\phi(t, x)-\phi_S(x;\tilde{\la}(t))|^2 dx \leq
	C \ep^2,\quad \forall t\in[0, T/\ep]
	\]
	for some  constant $C$ depending only on $m$, $T$, $p$ and $\la_0$. Here the covariant derivative $D$ is with respect to the given connection field $A^\ep$, that is, $D=\pa+i A^\ep$.
\end{proposition}
 
\subsection{Orthogonality condition and modulation equations}
Since the initial data are close to some soliton and  solution to \eqref{eq:CSF:fix} exists locally, we  decompose the solution as in \eqref{eq:decomp} for some curve $\la(t)\in \La$, which we write as
\[
\la(t)=(\om(t), \th(t), \eta(t)+u_0t, u(t)+u_0).
\]
Here we keep in mind that $\xi(t)=\eta(t)+u_0t$ and $u_0$ is parallel to $(0, 0, 1)$, that is the direction of the constant magnetic field.  We choose $\la(t)$ such that the orthogonality  \eqref{orthcond} holds. Differentiate the equation \eqref{orthcond}  with respect to the time variable $t$. We conclude  that the orthogonality condition  \eqref{orthcond} holds if it holds initially and the curve $\la(t)$ satisfies
\begin{equation*}
	\begin{split}
		&\langle \pa_\la^2\phi_S\cdot\dot{\la}, \e^{i\Th}w\rangle_{dx}+\langle \pa_\la\phi_S, \pa_t\left(\e^{i\Th}w\right)\rangle_{dx}=
		\langle \pa_\la^2\psi_S\cdot\dot{\la},  \e^{i\Th}v\rangle_{dx}+\langle \pa_\la\psi_S, \pa_t\left(\e^{i\Th}v\right)\rangle_{dx}.
	\end{split}
\end{equation*}
In view of the the decomposition \eqref{eq:decomp} and the relation $\dot{\la}=\dot{\ga}+V(\la)$, we can show that the above equation is equivalent to
\begin{align}
	\notag
	&\left(\langle \pa_\la \psi_S,  \pa_\la \phi_S\rangle_{dx}-\langle \pa_\la\phi_S,  \pa_\la\psi_S\rangle_{dx}+
	\langle \pa_\la^2\phi_S, \e^{i\Th}w\rangle_{dx}-\langle \pa_\la^2\psi_S, \e^{i\Th}v\rangle_{dx}\right)\dot{\ga}\\
	\label{modeq0}
	&=\langle \pa_\la(V(\la)\pa_\la\psi_S), \e^{i\Th}v\rangle_{dx}-\langle \pa_\la\phi_S,  \phi_{tt}-\pa_\la\psi_S V(\la)\rangle_{dx},
\end{align}
in which  we have replaced $\langle \pa_\la V(\la)\cdot \pa_\la \phi_S, \e^{i\Th}w\rangle$ with $\langle \pa_\la V(\la)\cdot \pa_\la \psi_S, \e^{i\Th}v\rangle$ in view of the orthogonality condition \eqref{orthcond} as well as the relation $\psi_S=\pa_\la \phi_S\cdot V(\la)$. 

Now under the Lorentz gauge condition, we can write that 
\begin{equation}
	\label{defofH}
	\begin{split}
		\Box_{A^\ep}\phi+\pa_{tt}\phi-\Delta \phi= 2i (A^\ep)^\mu \pa_\mu \phi +i \pa^\mu A^\ep_\mu \phi -(A^\ep)^\mu A_\mu^\ep \phi=b^\mu \pa_\mu \phi +c\phi, 
	\end{split}
\end{equation}
in which we denote that 
\[
b^\mu =2i (A^\ep)^\mu,\quad c=-(A^\ep)^\mu (A^\ep)_{\mu}.
\]
Recall the identity \eqref{idenofphiS} for the soliton $\phi_S$  and observe that
$$Re(\e^{-i\Th}\pa_\la\phi_S)=\pa_\la|\phi_S|.$$
Integration by parts then implies that  
\begin{align*}
	\langle \pa_\la\phi_S, \Delta_x \phi\rangle_{dx}&=\langle \pa_\la\phi_S, \Delta_x\phi_S\rangle_{dx}+\langle\Delta_x \pa_\la\phi_S, \phi-\phi_S\rangle_{dx}\\
	&=\langle \pa_\la\phi_S, m^2\phi_S-|\phi_S|^{p-1}\phi_S+\pa_\la\psi_S\cdot V(\la)\rangle_{dx}+\langle \pa_\la\Delta_x \phi_S, \phi-\phi_S\rangle_{dx}\\
	&=\langle \pa_\la(\pa_\la\psi_S\cdot V(\la)), \phi-\phi_S\rangle_{dx}+\langle \pa_\la\phi_S, \pa_\la\psi_S\cdot V(\la)\rangle_{dx}\\
	&\quad -\langle \pa_\la\phi_S, -m^2\phi+|\phi_S|^{p-1}\phi+(p-1)|\phi_S|^{p-2}\phi_S Re(\e^{-i\Th}(\phi-\phi_S))\rangle_{dx}.
\end{align*}
Now use the equation \eqref{defofH} to
replace $\pa_{tt}\phi$ in \eqref{modeq0}. The equation
\eqref{eq:CSF:fix} of $\phi$ then implies that the right hand side of \eqref{modeq0} can be written as
\begin{align}
	\label{eq:def4F}
	K(t;\la(t))=
	& -\langle b^\mu \pa_\la \phi_S, \pa_\mu\phi\rangle_{dx}-\langle \pa_\la\phi_S, c\phi\rangle_{dx}-\langle \pa_\la \phi_S, \e^{i\Th}\mathcal{N}(\la)\rangle_{dx}.
\end{align}
Here the nonlinearity is given by    
\begin{equation}
	\label{eq:def4:N}
	\mathcal{N}(\la)=\e^{-i\Th}\left(|\phi|^{p-1}\phi-|\phi_S|^{p-1}\phi_S\right)-f_\om ^{p-1}v-
	(p-1)f_\om^{p-1} Re(v),
\end{equation}
in which   $Re(v)$ is the real part of the complex valued function $v$ in view of  the decomposition \eqref{eq:decomp}.

To further simplify the above modulation equations, let
\begin{equation}
	\label{defofM}
	\begin{split}
		& M_0=\langle \pa_\la \psi_S,  \pa_\la \phi_S\rangle_{dx}-\langle \pa_\la\phi_S, \pa_\la\psi_S\rangle_{dx},\\
		&  M_1=\langle \pa_\la^2\phi_S, \e^{i\Th}w\rangle_{dx}-\langle \pa_\la^2\psi_S, \e^{i\Th}v\rangle_{dx}
	\end{split}
\end{equation}
be the $8\times 8$ matrices. The above computations imply that  the  modulation curve $\la(t)$ verifies the equation
\begin{equation}
	\label{modeq}
	(M_0+M_1)\dot{\ga}=K(t;\la(t)),\quad \dot{\la}=\dot\ga+V(\la), \quad \ga(0)=\la(0)
\end{equation}
such that $\la(0)$ verifies the orthogonality condition \eqref{orthcond} initially. This modulation equation is coupled to the nonlinear wave equation \eqref{eq:CSF:fix} under fixed connection field $A^{\ep}$. 
\subsection{Nondegeneracy of the Modulation Equations}
In view of the above modulation equation \eqref{modeq} for the modulation curve $\gamma$, to obtain estimates for $\ga(t)$, we show that the leading
coefficient matrix $M_0$ is non-degenerate while  $M_1$ is an error matrix relying on the error term $(v, w)$.

Since the initial data are close to some stable soliton with parameter $\la_0\in \Lambda$.  Let $\delta_0>0$ be a small positive constant such that the following set
\begin{equation}
	\label{defofLadel}
	\La_{\delta_0}=\{(\om, \th, \xi, u+u_0), \quad  |\om-\om_0|+ |u |<\delta_0\}
\end{equation}
is a subset of  $\La $. We first show that the leading matrix $M_0$ in the above modulation equation  \eqref{modeq} is non-degenerate. 
\begin{lemma}
	\label{nondegD}
	Let $M_0$ be the $8\times 8$ matrix defined in \eqref{defofM}. For
	$\la=(\om, \th, \xi, u)\in \La_{\delta_0} $, we have the bound
	\begin{equation*}
		|\det M_0|\geq C(\delta_0,\la_0) >0
	\end{equation*}
	for some positive constant $C(\delta_0,\la_0)$ depending only on $\delta_0$ and $\la_0$. In particular, $M_0$ is non-degenerate.
\end{lemma}
\begin{proof} 
This has been shown for example in \cite{yang4}. For readers' interest, we reprove it here. 
	By definition of translated solitons, we can compute the components of the matrix $M_0$
	\begin{align*}
		(M_0)_{\om\th}&=\pa_\om(\om \|f_\om\|_{L^2}^2),\quad (M_0)_{\om\xi}= \rho u \pa_\om G,\quad  (M_0)_{\xi u}=-\rho G(I+\rho^2 u\cdot u),\\
		(M_0)_{\om u}&= (M_0)_{\th \xi}=(M_0)_{\xi\xi}=(M_0)_{uu}=(M_0)_{\th u}= (M_0)_{\th\th}=0.
	\end{align*}
	Here we denote that 
	\begin{equation}
		\label{defofB}
		G(\omega)=\om^2 \|f_\om(x)\|_{L^2(\mathbb{R}^3)}^2+\frac{1}{3}\|\nabla_x f_\om(x)\|_{L^2(\mathbb{R}^3)}^2.
	\end{equation}
	Since the matrix  $M_0$ is anti-symmetric, we can compute that
	\[
	\det M_0=|\pa_\om(\om \|f_\om\|_{L^2}^2)|^2\cdot |\det (M_0)_{\xi
		u}|^2=|\pa_\om(\om \|f_\om\|_{L^2}^2)|^2 G^{6}\rho^{10}.
	\]
	In view of the scaling property of  the ground state  $f_{\om}(x)$ in Proposition \eqref{prop:ground:f}, we compute that
	\[
	\pa_\om(\om \|f_\om\|_{L^2}^2)=(m^2-\frac{6-2p}{p-1}\om^2)(m^2-\om^2)^{\frac{9-5p}{2(p-1)}}\|f\|_{L^2(\mathbb{R}^3)}^2,
	\]
	which is strictly negative for the stable solitons as 
	$
	\frac{p-1}{6-2p} < \frac{\om^2}{m^2} <1.
	$ 
	The lemma then follows as $\rho\geq 1$.
\end{proof}
\subsection{Initial Data}
We assume that the initial data are close to some stable soliton with parameter $\lambda_0$, which however may not verify the orthogonality condition \eqref{orthcond}. By using implicit functional theorem, we show in this section that up to an acceptable error, we can find a parameter $\la(0)$ such that the orthogonality condition holds initially. 
\begin{lemma}
	\label{datapre}
	Assume that for some $\la_0\in \La$ the initial data are close to the associated stable soliton
	\[
	\|\phi_0(x)-\phi_S(x;\la_0)\|_{H^1(\mathbb{R}^3)}+\|\phi_1(x)-\psi_S(x;\la_0)\|_{L^2(\mathbb{R}^3)}\leq \ep.
	\]
	Then there exists a positive constant $\ep_1(\la_0)$,
	depending only on $\la_0$, such that if $\ep < \ep_1(\la_0)$, then there exists
	$\la(0)\in  \La$ with the property that if
	\begin{equation*}
		\begin{cases}
			\phi_0(x)=\phi_S(x;\la(0))+\e^{i\Th(\la(0))}v(x;\la(0)),\\
			\phi_1( x)=\psi_S(x;\la(0))+\e^{i\Th(\la(0))}w(x;\la(0)),
		\end{cases}
	\end{equation*}
	then the orthogonality condition holds
	\begin{align*}
		&\langle \pa_\la\psi_S(x;\la(0)),
		\e^{i\Th(\la(0))}v(x;\la(0))\rangle_{dx}=\langle \pa_\la\phi_S(x;\la(0)),
		\e^{i\Th(\la(0))}w(x;\la(0))\rangle_{dx}.
	\end{align*}
	Moreover, we have
	\begin{align*}
		&|\la(0)-\la_0|\leq C(\la_0)\ep,\\
		& \|v(x;\la(0))\|_{H^1}+\|w(x;\la(0))\|_{L^2}\leq C(\la_0)\ep
	\end{align*}
	for some constant $C(\la_0)$ depending only on $\la_0$.
\end{lemma}
\begin{proof}
	Define a functional $\mathcal{F}: H^1\times L^2\times
	\mathbb{R}^8\rightarrow \mathbb{R}^8$ such that
	\begin{align*}
		\mathcal{F}(v, w, \la)=&\langle   \pa_\la\phi_S(\la;x),\phi_1(x)
		-\psi_S(\la;x)\rangle_{dx}-\langle  
		\pa_\la\psi_S(\la;x),\phi_0(x)-\phi_S(\la;x)
		\rangle_{dx}.
	\end{align*}
	In particular, we have $\mathcal{F}(0, 0, \la_0)=0$. Notice that
	\begin{align*}
		\mathcal{F}_\la(0, 0, \la_0)=\langle  \pa_\la\psi_S, \pa_\la \phi_S\rangle_{dx}-\langle   \pa_\la\phi_S, \pa_\la \psi_S\rangle_{dx}
		=M_0.
	\end{align*}
	In view of Lemma \ref{nondegD}, we can conclude that if $\ep$ is sufficiently small, depending only on $\la_0$, $\mathcal{F}_\la(0, 0, \la_0)$ is invertible.  Since the functional $\mathcal{F}$ is Lipschitz continuous in $(v, w)$, the implicit function theorem then implies that there exists $\la(0)\in \La$ satisfying
	the orthogonality condition \eqref{orthcond} and the estimates in the lemma hold.
\end{proof}
\subsection{Bootstrap assumptions}
The previous section indicates that we may assume the initial data are close to some stable soliton such that   the orthogonality condition holds. To prove Proposition \ref{prop:fix}, we use bootstrap argument. The local existence result for the nonlinear wave equation \eqref{eq:CSF:fix} is standard. In particular there is a short time solution $\phi(t, x)$. For the modulation equation \eqref{modeq}, the leading matrix $M_0$ is non-degenerate. We can bound the error matrix $M_1$ 
\begin{align*}
	|M_1|&=|\langle \pa_\la^2\phi_S, \e^{i\Th}w\rangle_{dx}-\langle \pa_\la^2\psi_S, \e^{i\Th}v\rangle_{dx}|\\
	&\les \|v\|_{L^2}+\|w\|_{L^2}\les \ep. 
\end{align*}
Here note that initially the center of the soliton $|\xi(0)|\les \ep$ in view of Lemma \ref{datapre}. In particular for sufficiently small $\ep$, we can solve the wave equation \eqref{eq:CSF:fix} and the modulation equation \eqref{modeq} locally. To extend the solution to the long time $T/\ep $, we expect that the modulation curve $\ga$ stays in the set $\La_{\delta_0}$ and the error matrix $M_1$ keeps small. To solve the wave equation \eqref{eq:CSF:fix}, we also need to control the center of the soliton. Recall that 
\[
\dot\la(t)=V(t)+\dot\gamma(t),\quad V(t)=(0, \frac{\om}{\rho}, u(t)+u_0, 0),\quad \xi(t)=\eta(t)+u_0t.
\]
We expect that $|\dot \gamma|$ is of size $\ep^2$. Therefore we can  conclude that 
\[
|\eta(t)|\leq |\xi(0)|+|u(t)-u(0)|+Ct^2\ep^2\leq |\xi(0)|+(C+1)T^2,\quad \forall t\leq T/\ep.
\]
We remark here that this is compatible with the main theorem   as if we scale it back to the
space $[0, T]\times \mathbb{R}^3$, the center of the particle  becomes $(t, \ep\eta+u_0 t)$ which is close to the straight line $(t, u_0 t)$. In addition to the assumption that $\la(t)\in \La_{\delta_0}$, we make the following bootstrap assumptions 
\begin{align}
	&\label{baxi}
	|\eta(t)|\leq 2C_2, \quad \forall t\in[0, T/\ep],\\
	&\label{bawv}
	|u|+\|w(t, x)\|_{L^2(\mathbb{R}^3)}+\|v(t, x)\|_{H^1(\mathbb{R}^3)}\leq 2\delta_1,\quad \forall t\in[0, T/\ep]
\end{align}
for some positive constants $C_2$, $\delta_1$ which will be fixed later. Without loss of generality, we assume that $$C_2 >1, \quad \delta_1  <1, \quad C_2\delta_1\leq 1, \quad C_2^4\ep <1.$$  
\subsubsection{Estimates for the Modulation Curve}
Under the above bootstrap assumptions, we first control the modulation curve. 
We have shown that the leading matrix $M_0$ is non-degenerate as long as $\la\in \La_{\delta_0}$. To obtain necessary estimate for the modulation curve $\ga(t)$ by
using the modulation equation, we need  to show that $M_1$,  
$K(t;\la(t))$ are error terms. By the definition of $M_1$, under the bootstrap assumption \eqref{bawv}, we can bound that 
\begin{align}
	\label{D1}
	|M_1|\les (\|v\|_{L^2}+\|w\|_{L^2})(\|\pa_\la \phi_S\|_{L^2}+\|\pa_\la \psi_S\|_{L^2})\leq C \delta_1
\end{align}
for some constant $C$ depending only on $\delta_0$. We now can   choose $\delta_1$ as follows in order to make the matrix $M_0+M_1$ non-degenerate: For fixed $\delta_0$, let  $\delta_1$
be sufficiently small such that
\begin{equation*}
	\delta_1\leq
	C^{-1}\frac{1}{10} C(\delta_0, \la_0)^{\frac{1}{8}}.
\end{equation*}
Here $C(\delta_0, \la_0)$ is the lower bound  appeared in Lemma  \ref{nondegD}, which is independent of $\ep$ and $C_2$. Now we are ready to estimate the modulation curve $\dot \ga$. In particular we can solve the ODE \eqref{modeq} and the modulation curve $\gamma(t)$ verifies the bound 
\[
|\dot \ga|\les |K(t; \la(t))|. 
\]
Recall that $K$ consists of two types of nonlinear terms: the nonlinear interaction of the solution with the connection field and the nonlinear terms arising from the potential. For  complex number $\zeta$, let
\begin{equation*}
	N(\zeta)= \frac{1}{p+1}|\zeta|^{p+1}.
\end{equation*}
In particular, we have $\pa_{\bar \zeta}N=\f12
|\zeta|^{p-1}\zeta$, where $\bar \zeta$ is the complex conjugate of $\zeta$. In view of the 
decomposition for the solution, we can write the nonlinearity 
\[
\mathcal{N}(\la)=2\pa_{\bar \zeta}N (f_\om+ v)-2\pa_{\bar \zeta}N(f_\om)-2\pa_{\bar \zeta} \pa N(f_\om)\cdot
v.
\]
Here $\pa N\cdot v=\pa_\zeta N v+\pa_{\bar \zeta}N \bar v$. We have the following bounds for the nonlinearity. 
\begin{lemma}
	\label{lemnonlinear}
	Assume $\phi$ decomposes as \eqref{eq:decomp}. For all $p\geq 2$, we have
	\begin{align}
		\label{Nvpt}
		&|\mathcal{N}(\la)|\les |v|^2+|v|^p,\\
		\label{nonlinq} &\left|N(f_\om+v)-N(f_\om)-\pa N(f_\om)v-\f12
		v\pa^2N(f_\om)v\right|\les |v|^3+|v|^{p+1}.
	\end{align}
\end{lemma}
\begin{proof}
	We can write $\mathcal{N}(\la)$ as an integral
	\begin{align*}
		\mathcal{N}(\la)=4\int_{0}^{1}\int_{0}^{1}s v  \pa^2\pa_{\bar
			\zeta}N(f_\om+ts v)vdtds.
	\end{align*}
	Since $p\geq 2$, we can show that
	\[
	|\pa^2 \pa_{\bar \zeta}N(f_\om+ts v)|\les 1+|f_\om + ts  v|^{p-2}\les 1+| v|^{p-2}.
	\]
	Thus the bound  \eqref{Nvpt} holds. The second inequality \eqref{nonlinq} follows in a similar way.
\end{proof}
We are now ready to show that the nonlinearity  $K(t;\la(t))$ is in fact
higher order error terms. 
\begin{proposition}
	\label{lemDF}
	Under the bootstrap assumptions \eqref{baxi}, \eqref{bawv}, we have
	\begin{align*}
		|K|&\les C_2\ep^2+\|v\|_{H^1}^2.
	\end{align*}
\end{proposition}
\begin{proof}
	Recall the definition for $K$. For the nonlinear terms, in view of the above Lemma \ref{lemnonlinear} and the exponential decay of the soliton, we can show that 
	\begin{align*}
		|\langle \pa_\la \phi_S, \e^{i\Th}\mathcal{N}(\la)\rangle_{dx}|\les \int_{\mathbb{R}^3}|\pa_\la\phi_S| (|v|^2+|v|^{p})dx\les \|v\|_{H^1}^2
	\end{align*}
	by using Sobolev embedding. For the other two linear terms arising from the connection field, decompose them round the soliton and the constant connection field $A^\ep_b$. We first can show that 
	\begin{align*}
		|c\phi|&=|(A^\ep)^\mu (A^\ep)_\mu (\phi_S+e^{i\Theta}v)|\les (|\tilde{A}^\ep|^2+   |A^\ep_b|^2 )(|\phi_S|+|v|),\\ 
		|b^\mu \pa_\mu\phi|&=|(A^\ep)^0(\psi_S+e^{i\Theta}w)+(A^\ep)^j\pa_j(\phi_S+e^{i\Theta}v)|\\ 
		& \les |\tilde{A}^\ep|(|\psi_S|+|w|+|v|+|\nabla v|+|\nabla \phi_S|)+\ep^2 |\Omega_{12}\phi_S|+   |A^\ep_b||\nabla v|.  
	\end{align*}
	Here $\Omega_{12}=x_1\pa_2-x_2\pa_1$. By the definition of $\phi_S$, we compute that
	\begin{align*}
		\Omega_{12}\phi_S= e^{i\Theta} (\nabla f_{\omega}(z)\Omega_{12}z+i\Omega_{12}\Theta f_{\omega})=e^{i\Theta} (\nabla f_{\omega}(z)\Omega_{12}z-\omega i\Omega_{12}z\cdot  (u+u_0) f_{\omega}).
	\end{align*}
	Note that $$\Omega_{12}(x-\xi)=(-x_2,x_1, 0),$$ which is orthogonal to the vector field $u_0$ (parallel to the constant magnetic field $(0, 0, 1)$). Recall that $\xi=\eta+u_0t$. We thus can bound that 
	\begin{align}
		\label{eq:Omphis}
		|\Omega_{12}\phi_S|\les (|x_1|+|x_2|) (|f_{\omega}|+|\nabla f_{\omega}|)\les (|z|+|\eta|)(|f_{\omega}|+|\nabla f_{\omega}|) .
	\end{align}
	In view of the bootstrap assumption \eqref{baxi} and combining all the above computations, we then can show that   
	\begin{align*}
		& |\langle b^\mu \pa_\la \phi_S, \pa_\mu\phi\rangle_{dx}|+|\langle \pa_\la\phi_S, c\phi\rangle_{dx}|\\ 
		& \les  (\|w\|_{L^2}+\|v\|_{H^1})(\| (|f_{\omega}|+|\nabla f_{\omega}|) (|\tilde{A}^\ep|^2 +\ep^4 (x_1^2+x_2^2)+\ep^2(|x_1 |+|x_2|)\|_{L^2}) \\
		&\quad +\ep^2\|(|z|+|\eta|)(|f_{\omega}|+|\nabla f_{\omega}|)^2\|_{L^1}+\||\tilde{A}^\ep| |\pa_\lambda\phi_S|(|\psi_S|+|\nabla\phi_S|+|\phi_S|)\|_{L^1}\\  
		&\les  C_2\ep^2+(\|w\|_{L^2}+\|v\|_{H^1})(\ep^2+\ep^4 C_2^2+\ep^2C_2)+\ep^2 \| (1+|z|)^{\frac{1}{2}}f_\om(z)(f_\om(z)+|\nabla f_\om(z)|)\|_{L^1_z}\\ 
		&\les C_2\ep^2.
	\end{align*} 
	Here the bound for $\tilde{A}^\ep$ follows from the bootstrap assumption \eqref{eq:cond4:A} and the estimate \eqref{eq:tildeA:bd}. We also used the fact that 
	\[
	|x-\xi(t)|\les |z|
	\]
	and the ground state $f_{\om}(z)$ decays exponentially. 
	\end{proof}
The above estimate then implies that the modulation curve stays close to the integral curve of $V(\la)$.
\begin{corollary}
	\label{corcontrga}
	Suppose $\la(t)\in \La_{\delta_0}$ and $\eta(t)$, $w, v$ satisfy the bootstrap assumptions \eqref{baxi},
	\eqref{bawv}. Then we have
	\begin{equation*}
		|\dot{\ga}|\les |K|\les C_2\ep^2+\|v\|_{H^1}^2.
	\end{equation*}
\end{corollary}
This corollary shows that as long as the radiation term $v$ in the decomposition \eqref{eq:decomp} of the solution is small, we can solve the modulation equations \eqref{modeq} and obtain estimates for the modulation curve $\dot \ga(t)$. The modulation equations are used to guarantee the orthogonality condition \eqref{orthcond}. Next, we show that under the orthogonality condition, the energy $\|v\|_{H^1}+\|w\|_{L^2}$ of the radiation term $(v, w)$ is small.
\subsubsection{Energy Decomposition}
First define the total charge
\begin{equation}
	\label{defcharge}
	Q(t):=-\int_{\mathbb{R}^3}\langle i \pa_t\phi, \phi\rangle dx.
\end{equation}
In view of the decomposition \eqref{eq:decomp}, we have 
\begin{align*}
	Q(t) &:=-\langle i\psi_S, \phi_S\rangle_{dx}-\langle iw, v\rangle_{dx}-\langle i\psi_S, \e^{i\Th}v\rangle_{dx}-\langle i\e^{i\Th}w, \phi_S\rangle_{dx}.
\end{align*} 
Observe that $\pa_\th\phi_S=i\phi_S, \pa_\th\psi_S=i\psi_S$. The orthogonality condition \eqref{orthcond} indicates that 
\[
\langle i\psi_S, \e^{i\Th}v\rangle_{dx}+\langle i\e^{i\Th}w, \phi_S\rangle_{dx}=\langle \pa_\th\psi_S, \e^{i\Th}v\rangle_{dx}-\langle \pa_\th\phi_S, \e^{i\Th}w\rangle_{dx}=0.
\]
We can compute the soliton part
\[
\langle i\psi_S, \phi_S\rangle_{dx}=\int_{\mathbb{R}^3}-\rho\om f_\om^2 dx=-\om\|f_\om\|_{L^2},
\]
which then leads to the decomposition of the total charge
\begin{equation}
	\label{Qdecomp}
	Q(t)=\om\|f_\om\|_{L^2}^2-\langle iw, v\rangle_{dx}.
\end{equation}
Next for $k=1, 2, 3$, define the angular momentum 
\begin{equation}
	\label{eq:def:Pik}
	\Pi_k(t)=\int_{\mathbb{R}^3} \langle \pa_t\phi, \pa_k\phi\rangle dx.
\end{equation}
Similarly, we decompose the angular momentum around the soliton 
\begin{equation*}
	\begin{split}
		\Pi_k(t) 
		&=\int_{\mathbb{R}^3}\langle\psi_S+\e^{i\Th}w, \pa_k(\phi_S+\e^{i\Th}v)\rangle dx\\ 
		&=\langle\psi_S, \pa_k\phi_S\rangle_{dx}-\langle \pa_k\psi_S,  \e^{i\Th}v\rangle_{dx}+\langle\e^{i\Th}w, \pa_k\phi_S\rangle_{dx}+
		\langle\e^{i\Th}w, \pa_k(\e^{i\Th}v)\rangle_{dx}.
	\end{split}
\end{equation*}
For the soliton part, note that  $dz=\rho dx$. We can compute
\begin{align*}
	\langle\psi_S, \pa_k\phi_S\rangle_{dx}
	&=\langle\e^{i\Th}(i\rho \om f_\om-\rho (u+u_0)\cdot \nabla_z f_\om), \e^{i\Th}(-i\rho\om f_\om (u+u_0)_k+\nabla_z f_\om \cdot\frac{\pa z}{\pa x_k})\rangle_{dx}\\ 
	&=-\int_{\mathbb{R}^3}\rho^2\om^2 f_\om^2 (u+u_0)_k+ \rho (u+u_0)\cdot \nabla_z f_\om \nabla_z f_\om \cdot \frac{\pa z}{\pa x_k}dx\\ 
	&=-\rho G (u+u_0)_k.
\end{align*} 
Here  $G$ is given in line \eqref{defofB}.

For the crossing term,  in view of the orthogonality condition \eqref{orthcond},  we can show that  
\begin{align*}
	\langle \pa_k\psi_S,  \e^{i\Th}v\rangle_{dx}-\langle\e^{i\Th}w, \pa_k\phi_S\rangle_{dx}  
	= -\langle \pa_{\xi_k}\psi_S,  \e^{i\Th}v\rangle_{dx}+\langle\e^{i\Th}w, \pa_{\xi_k}\phi_S\rangle_{dx}   =0.
\end{align*}
Therefore we end up with the decomposition for the angular momentum
\begin{equation}
	\label{Pidecomp}
	\begin{split}
		\Pi_k(t)=-\rho G (u+u_0)_k+\langle\e^{i\Th}w, \pa_k(\e^{i\Th}v)\rangle_{dx}.
	\end{split}
\end{equation}
Now we define the standard energy 
\begin{align}
	\label{eq:def4:Pi0}
	\Pi_0(t)=\frac{1}{2}\int_{\mathbb{R}^3}|\pa_t\phi|^2+|\nabla\phi|^2+m^2|\phi|^2-\frac{2}{p+1}|\phi|^{p+1}dx. 
\end{align}
For the quadratic terms, in view of the decomposition \eqref{eq:decomp}, we have 
\begin{equation*}
	\begin{split}
		&\int_{\mathbb{R}^3}\langle  \pa\phi, \pa\phi\rangle  +m^2\langle\phi, \phi\rangle dx\\
		&=\int_{\mathbb{R}^3} | \nabla (\phi_S+\e^{i\Th} v)|^2 +|\psi_S+\e^{i\Th} w |^2 +m^2| \phi_S+\e^{i\Th} v|^2 dx\\
		&=\int_{\mathbb{R}^3}|\nabla\phi_S|^2+|\psi_S|^2+m^2|\phi_S|^2 +m^2 | v|^2 +|w|^2 +
		|\nabla (\e^{i\Th}v)|^2 \\
		&\quad    +2\langle \nabla \phi_S, \nabla(\e^{i\Th}v)\rangle +2\langle\e^{i\Th}w, \psi_S\rangle +2m^2\langle\phi_S,
		\e^{i\Th}v\rangle dx.
	\end{split}
\end{equation*}
For the nonlinear term, we expand it up to the second order
\begin{align*}
	-\int_{\mathbb{R}^3}N(\phi)dx&=-\int_{\mathbb{R}^3}N(f_\om+ v)-N(f_\om)-\pa N(f_\om) v-\f12 v\pa^2N(f_\om)v\ dx\\
	&\quad - \int_{\mathbb{R}^3}\frac{1}{p+1}f_\om^{p+1}+f_\om^p  v_1+\f12 f_\om^{p-1} (|v|^2+(p-1)|v_1|^2)\ dx\\
	&=- \int_{\mathbb{R}^3}\frac{1}{p+1}f_\om^{p+1}+f_\om^p
	v_1+\f12f_\om^{p-1}(|v|^2+(p-1)|v_1|^2)\ dx+O(\|v\|_{H^1}^3).
\end{align*}
Here the higher order nonlinear term is bounded by using Sobolev embedding and  Lemma \ref{lemnonlinear}. We also use the assumption that $p\geq 2$.  
Combining these together, we obtain that 
\begin{align*}
	\Pi_0(t)&=\f12\int_{\mathbb{R}^3}|\nabla\phi_S|^2+|\psi_S|^2+m^2f_\om^2-\frac{2}{p+1}f_\om^{p+1}
	+m^2 |v|^2
	+|\nabla(\e^{i\Th}v)|^2  dx\\
	&+\langle \nabla \phi_S, \nabla(\e^{i\Th}v)\rangle_{dx}+\f12\langle w,
	w\rangle_{dx}+\langle\e^{i\Th}w, \psi_S\rangle_{dx}+m^2\langle\phi_S,
	\e^{i\Th}v\rangle_{dx} \\
	& -\langle f_\om^p, v_1\rangle_{dx}-\f12\int_{\mathbb{R}^3}f_\om^{p-1}|v|^2+(p-1)f_\om^{p-1}v_1^2 dx+O(\|v\|_{H^1}^3).
\end{align*}
For the soliton part, by using the scaling property for the ground state in Proposition \ref{prop:ground:f}, we compute that 
\begin{align*}
	& \f12\int_{\mathbb{R}^3}|\nabla\phi_S|^2+|\psi_S|^2+m^2f_\om^2-\frac{2}{p+1}f_\om^{p+1}dx\\
	=&\f12  \int_{\mathbb{R}^3}\rho^2\om^2|u+u_0|^2f_\om^2+|\nabla_x f |^2+2\rho^2\om^2 f_\om^2+ m^2f_\om^2-\frac{2}{p+1}f_\om^{p+1} dx \\
	=& \f12\int_{\mathbb{R}^3}2\rho^2\om^2 f_\om^2+(m^2-\om^2)f_\om^2+(\frac{1}{3}+\frac{2}{3}\rho^2)|\nabla_z f|^2-\frac{2}{p+1}f_\om^{p+1}dx\\
	=& \rho\left(\om^2\|f_\om\|_{L^2}^2+\frac{1}{3}\|\nabla_z f_\om\|_{L^2}^2\right)=\rho G.
\end{align*}
For the first order terms of $(v, w)$, in view of the  identity \eqref{idenofphiS} and the orthogonality
condition \eqref{orthcond}, we show that 
\begin{align*}
	& \langle \nabla \phi_S, \nabla (\e^{i\Th}v)\rangle_{dx}+\langle\psi_S, \e^{i\Th}w\rangle_{dx}+\langle m^2\phi_S, \e^{i\Th}v\rangle_{dx}-\langle f_\om^p, v_1\rangle_{dx}\\
	&= \langle\psi_S, \e^{i\Th}w\rangle_{dx}+\langle-\Delta \phi_S+m^2\phi_S-\e^{i\Th}f_\om^p,\e^{i\Th}v\rangle\\
	&= \langle \pa_\la \phi_S\cdot V(\la), \e^{i\Th}w\rangle-\langle \pa_\la \psi_S \cdot V(\la), \e^{i\Th}v\rangle =0.
\end{align*}
Combining all the above estimates, we therefore have the following decomposition for the standard energy 
\begin{equation}
	\label{Hdecomp}
	\begin{split}
		\Pi_0(t)=\rho G+\f12\int_{\mathbb{R}^3}m^2 |v|^2
		+|\nabla(\e^{i\Th}v)|^2 +|w|^2-f_\om^{p-1}|v|^2 -(p-1)f_\om^{p-1}v_1^2 dx +O( \|v\|_{H^1}^3 ) .
	\end{split}
\end{equation}
\subsubsection{Weighted energy estimates in the exterior region}
\label{sec:weightE}
The main difficulty of  the long time dynamics for charged scalar fields on constant electromagnetic field is that the background connection field  $A^\ep_b$ grows linearly at spatial infinity. The associated error terms are of the form $ \||x| v\|_{L^2}$ which could not be bounded by the standard energy of the remainder terms $(v, w)$. This is the main reason that in the works
\cite{Stuart09:MKG}, \cite{Eamonn06:MKG}, the background connection field $A^\ep_b$ was assumed to be uniformly bounded. Note that the soliton decays exponentially, such weighted error terms can be controlled if we have the associated weighted energy estimates for the full solution $\phi$. We use the ideas in the works \cite{Yang:mMKG}, \cite{YangYu:MKG:smooth} for the study of massive Maxwell-Klein-Gordon system and show that outside of large forward light cone, the weighted energy estimate for the scalar field in uniformly bounded. 

\bigskip

Since initially the data are close to some soliton, which decays exponentially, for sufficiently large $R_0>0$, depending only on $\la_0$, we can require that 
\begin{align*}
	\int_{|x|\geq R_0} (1+|x|^2)(|\nabla\phi_0|^2+|\phi_0|^2+|\phi_1|^2) dx\leq 10\ep^2,
\end{align*}
which follows, in view of the assumption \eqref{eq:ID:sca:w0} on the initial data in Proposition \ref{prop:fix}, by choosing $R_0$ such that 
\begin{align*}
	\int_{|x|\geq R_0} (1+|x|^2)(|\nabla\phi_S(x;\la_0)|^2+|\phi_S(x;\la_0)|^2+|\psi_S(x;\la_0)|^2) dx\leq \ep^2. 
\end{align*}
The solution to the equation \eqref{eq:CSF:fix} exists locally in the exterior region $\{ t+R_0\leq |x|, \quad t\leq T/\ep\}$. We  use bootstrap argument to  show that the energy through the out going null hypersurface and the constant time hypersurface in the exterior region is uniformly bounded. For this purpose, let 
\[
L=\pa_t+\pa_r,\quad r=|x|,\quad \D=(D_1,D_2, D_3)- r^{-1}x D_r.
\]
Here the covariant derivative associated to the connection field $A^\ep$ is defined by $D=\pa+i A^\ep$. 

For $0\leq t\leq T/\ep$ and $R\geq R_0$, let $\Sigma_{t, R}$ be the constant time hypersurface 
\[
\Sigma_{s, R}=\{(t, x), t=s,\quad |x|\geq R+s\}
\]
and $\mathcal{C}_{t, R}$ be the out going null cone
\[
\mathcal{C}_{s, R}=\{ (t, x),\quad 0\leq t\leq s,\quad |x|-s=R\}. 
\]
For all $0\leq t\leq T/\ep$ and $R\geq R_0$, denote  
\[
\mathcal{E}_{t, R}=\int_{\mathcal{C}_{t, R}} |D_L\phi|^2+| \D\phi|^2+m^2|\phi|^2 d\sigma +\int_{ \Sigma_{t, R}}|D\phi|^2+m^2|\phi |^2dx.
\]
Here $d\sigma$ is the surface measure on the out going null hypersurface $\mathcal{C}_{t, R}$. Obviously on the initial hypersurface, the assumption on the initial data implies that 
\[
\mathcal{E}_{0, R}\les R^{-2} \ep^2,\quad \forall R\geq R_0. 
\] 
For solution $\phi$ to the equation \eqref{eq:CSF:fix}, define the associated energy momentum tensor 
\[
T_{\mu\nu}[\phi]=\langle D_\mu\phi, D_\nu\phi\rangle-\f12 m_{\mu\nu}(\langle D^{\ga}\phi, D_\ga \phi\rangle+2\mathcal{V}(\phi))
\]
with  the potential 
\[
\mathcal{V}(\phi)=\frac{m^2}{2}|\phi|^2-\frac{1}{p+1}|\phi|^{p+1}=\frac{m^2}{2}|\phi|^2-N(\phi)
\]
and the Minkowski metric $m_{\mu\nu}$  on the Minkowski space $\mathbb{R}^{1+3}$. For any vector field $Y$, we have the energy identity
\begin{equation}
	\label{enerestfor}
	\pa^{\mu}( T_{\mu\nu}[\phi]Y^\nu)=  T^{\mu\nu}[\phi]\pi^{Y}_{\mu\nu}+ \langle\Box_{A^\ep} \phi-m^2\phi+|\phi|^{p-1}\phi, D_Y \phi\rangle+\l Y^\nu D^{\gamma}\phi, i F_{\gamma\nu}^\ep \phi \r,
\end{equation}
where $\pi^{Y}_{\mu\nu}=\f12 \mathcal{L}_Y m_{\mu\nu}$ is the deformation tensor of the vector field $Y$ along the flat Minkowski metric $m_{\mu\nu}$. Apply the Killing vector field $\pa_t$ to the region bounded by $\Sigma_{t, R}$, $\mathcal{C}_{t, R}$ and the initial hypersurface $\Sigma_{0, R}$. We obtain the energy identity 
\begin{align*}
	&\int_{\Sigma_{t, R}}|D\phi|^2+m^2|\phi|^2-\frac{2}{p+1}|\phi|^{p+1}dx +2\int_{\mathcal{C}_{t, R}} |D_L\phi|^2+|\D\phi|^2 +m^2|\phi|^2-\frac{2}{p+1}|\phi|^{p+1} d\sigma\\ 
	&=\int_{\Sigma_{0, R}}|D\phi|^2+m^2|\phi|^2-\frac{2}{p+1}|\phi|^{p+1}dx  -2\int_0^t \int_{\Sigma_{s, R}} \l   D^{\gamma}\phi, i F_{\gamma 0}^\ep \phi \r dxds
\end{align*}
for solution $\phi $ to the equation \eqref{eq:CSF:fix}.

For $0\leq t\leq T/\ep$, $R\geq R_0$, we make the following bootstrap assumption on the energy $\mathcal{E}_{t, R}$
\begin{equation}
	\label{eq:bt:ex}
	\mathcal{E}_{t, R}\leq 2C \mathcal{E}_{0, R}
\end{equation}
for some constant $C$ depending only on $p$, $T$ and the mass $m>0$. 
For the nonlinear term, we use Gagliardo-Nirenberg interpolation inequality to bound that 
\begin{align*}
	\int_{\Sigma_{t, R}} |\phi|^{p+1}dx +\int_{\mathcal{C}_{t, R}}   |\phi|^{p+1} d\sigma\les \|D\phi\|_{L^2}^{ \frac{3(p-1)}{2}}\|\phi\|_{L^2}^{\frac{5-p}{2} }\les  C^{\frac{p+1}{2}}(\mathcal{E}_{0, R})^{p+1}.
\end{align*}
Here we used the fact that $|\pa |\phi||\leq |D\phi|$. For the nonlinear term arising from the connection field, by using the assumption \eqref{eq:cond4:A} for $A^\ep$ and the above bootstrap assumption \eqref{eq:bt:ex}, we can estimate that 
\begin{align*}
	|\int_0^t \int_{\Sigma_{s, R}} \l   D^{\gamma}\phi, i F_{\gamma 0}^\ep \phi \r dxds|\les \ep^2 \int_0^t |\phi||D\phi| dx ds\les \ep^2 \int_0^t \mathcal{E}_{s, R} ds\les \ep^2 C \mathcal{E}_{0, R} T/\ep \les C \ep \mathcal{E}_{0, R}.  
\end{align*}
Therefore the above energy identity leads to 
\begin{align*}
	\mathcal{E}_{t, R} \leq C_1\left(  \mathcal{E}_{0, R}+C^{\frac{p+1}{2}}\mathcal{E}_{0, R}^{p+1}+  C \ep \mathcal{E}_{0, R}\right)
\end{align*}
for some constant $C_1$ depending only on $p$, $T$ and $m>0$. Let $C=2C_1$ and for sufficiently small $\ep>0$, depending only on $p$, $T$ and $m$, we conclude that 
\begin{align*}
	\mathcal{E}_{t, R}\leq C \mathcal{E}_{0, R},\quad \forall R\geq R_0,\quad 0\leq t\leq T/\ep. 
\end{align*}
We thus improved the bootstrap assumption \eqref{eq:bt:ex}, which in particular implies that 
\begin{align*}
	\int_{\Sigma_{t, R}} |D\phi|^2+m^2|\phi|^2 dx\les \mathcal{E}_{0, R},\quad \forall t\leq T/\ep, \quad R\geq R_0.
\end{align*}
Integrate in terms of $R$. We derive that 
\begin{align*}
	\int_{R_1}^\infty \int_{\Sigma_{t, R}} |D\phi|^2+m^2|\phi|^2 dx dR & =\int_{\Sigma_{t, R_1}} (|x|-t-R_1) (|D\phi|^2+m^2|\phi|^2)dx\\ 
	&\les \int_{\Sigma_{0, R_1}} (|x|-R_1) (|D\phi|^2+m^2|\phi|^2)dx\\ 
	&\les  \int_{\Sigma_{0, R_1}} (1+|x|) (|D\phi|^2+m^2|\phi|^2)dx+(1+R_1)\mathcal{E}_{0, R_1}. 
\end{align*}
In particular we have 
\begin{align*}
	\int_{\Sigma_{t, R_1}} (1+|x|) (|D\phi|^2+m^2|\phi|^2)dx \les \int_{\Sigma_{0, R_1}} (1+|x|) (|D\phi|^2+m^2|\phi|^2)dx+(1+T/\ep +R_1) \mathcal{E}_{0, R_1}.
\end{align*}
Integrate in $R_1$ again. Similarly we can derive that 
\begin{align*}
	&\int_{\Sigma_{t, R_1}} (1+|x|^2) (|D\phi|^2+m^2|\phi|^2)dx \\ 
	&\les \int_{\Sigma_{0, R_1}} (1+|x|^2) (|D\phi|^2+m^2|\phi|^2)dx + (1+T/\ep +R_1)^2 \mathcal{E}_{0, R_1}\\ 
	&+(1+T/\ep +R_1) \int_{\Sigma_{0, R_1}} (1+|x|) (|D\phi|^2+m^2|\phi|^2)dx\\ 
	&\les \ep^2 +(1+T/\ep +R_1)\ep^2 (1+R_1)^{-1}+ (1+T/\ep +R_1)^2 \ep^{2}(1+R_1)^{-2}\\
	&\les 1
\end{align*}
for all $0\leq t\leq T/\ep$ and $R_1\geq R_0$. To summarize, we have shown in this section that 
\begin{align}
	\label{eq:weighted:phi:ex}
	\int_{\Sigma_{t, R_0}} (1+|x|^k) (|D\phi|^2+m^2|\phi|^2)dx \les    \ep^{2-k},\quad k=1, 2. 
\end{align}
\subsubsection{Energy Estimates for the Full Solution}
For solution to the wave equation \eqref{eq:CSF:fix}, the energies defined in the previous section are not conserved. Based on the weighted mass bound in the exterior region, we demonstrate in this section that they are almost conserved. 
\begin{proposition}
	\label{propcon}
	Assume that the connection field $A^\ep$ verifies the condition  \eqref{eq:cond4:A}. Under the bootstrap assumptions \eqref{baxi}, \eqref{bawv}, the energies satisfy the following bound   
	\begin{align*}
		|Q(t)-Q(0)|&\les  \ep^2+\ep \|v\|_{L^2}^2,\\ 
		\left|\Pi_k(t)-\Pi_k(0)+ \int_0^t \int_{\mathbb{R}^3}    \langle   i \pa_k(A^\ep)^\mu \pa_\mu\phi_S , \phi_S\rangle   dx ds\right| &\les \ep^2+\ep^2 \int_0^t \|v\|_{H^1}+\|w\|_{L^2} ds+\ep \|v\|_{L^2}^2 \\ 
		\left|\Pi_0(t)-\Pi_0(0)+\int_0^t \int_{\mathbb{R}^3}    \langle   i \pa_t(A^\ep)^\mu \pa_\mu\phi_S , \phi_S\rangle dx ds\right|  
		&\les \ep^2+  \ep^2 \int_0^t\|v\|_{H^1}+\|w\|_{L^2}ds\\ 
		&\quad + \ep(|u(t)|+\|v(t, x)\|_{H^1})+\ep \|v\|_{L^2}^2
	\end{align*}
	for all $0\leq t\leq T/\ep$. Here we used $\pa_0\phi_S=\pa_t\phi_S$ to stand for $\psi_S$. 
	\end{proposition}
\begin{proof}
	By the definition of the total charge, in view of the equation \eqref{eq:CSF:fix} and the definition for $b^\mu$, $c$ in  \eqref{defofH}, we have 
	\begin{align*}
		\pa_t Q(t)&=-\int_{\mathbb{R}^3}\l i \pa_{tt} \phi, \phi  \r dx\\ 
		&=\int_{\mathbb{R}^3}\l i  (m^2\phi-|\phi|^{p-1}\phi-\Delta\phi-b^\mu\pa_\mu\phi-c\phi), \phi  \r dx\\
		&=\int_{\mathbb{R}^3}\l   2 (A^\ep)^\mu\pa_\mu\phi , \phi  \r dx \\ 
		&= \int_{\mathbb{R}^3}    (A^\ep)^0 \pa_t|\phi|^2-\pa_k (A^\ep)^k |\phi|^2  dx\\ 
		&=\pa_t \int_{\mathbb{R}^3}    (A^\ep)^0  |\phi|^2   dx.
	\end{align*}
	Here we used the Lorentz gauge condition $\pa^\mu (A^\ep)_{\mu}=0$. In particular we have 
	\begin{align*}
		|Q(t)-Q(0)|\leq \int_{\mathbb{R}^3}    |(A^\ep)^0(t, x)|  |\phi(t, x)|^2   dx + \int_{\mathbb{R}^3}    |(A^\ep)^0(0, x)|  |\phi(0, x)|^2   dx.
	\end{align*}
	Using the assumptions \eqref{eq:cond4:A} and the estimate \eqref{eq:tildeA:bd} for the connection field $A^\ep$, we can bound that 
	\begin{align*}
		|Q(t)-Q(0)|&\les \ep^2\| \phi_0(x)\|_{L^2}^2+\ep^2 \| (1+|x-\xi(t)|)^{\frac{1}{4}} \phi_S(x;\la(t))\|_{L^2}^2+ \ep \|v(x;\la(t))\|_{L^2}^2 \\ 
		& \les  \ep^2+\ep \|v\|_{L^2}^2+\ep^2\|(1+|z|)^{\frac{1}{4}}f_\om(z)\|_{L^2}^2\\ 
		&\les  \ep^2+\ep \|v\|_{L^2}^2. 
	\end{align*}
For the angular momentum, similarly we have 
\begin{align*}
	\pa_t\Pi_k(t)&=\int_{\mathbb{R}^3} \langle \pa_{tt}\phi, \pa_k\phi\rangle + \langle \pa_{t}\phi, \pa_k \pa_t \phi\rangle dx \\ 
	&= \int_{\mathbb{R}^3} \langle  -m^2\phi+|\phi|^{p-1}\phi+\Delta\phi+b^\mu\pa_\mu\phi+c\phi, \pa_k\phi\rangle   dx \\ 
	&= \int_{\mathbb{R}^3} \langle   2i (A^\ep)^\mu\pa_\mu\phi +c\phi, \pa_k\phi\rangle   dx.
\end{align*}
By using the Lorentz gauge condition $\pa^\mu A_\mu^\ep =0$, we note that 
\begin{align*}
	&\langle   i (A^\ep)^\mu\pa_\mu\phi , \pa_\nu\phi\rangle  \\ 
	& =\pa_{\nu} \langle   i (A^\ep)^\mu\pa_\mu\phi , \phi\rangle -   \langle   i \pa_\nu(A^\ep)^\mu \pa_\mu\phi , \phi\rangle -\langle   i (A^\ep)^\mu \pa_\nu\pa_\mu\phi , \phi\rangle \\ 
	&=\pa_\nu \langle   i (A^\ep)^\mu\pa_\mu\phi , \phi\rangle -   \langle   i \pa_\nu(A^\ep)^\mu \pa_\mu\phi , \phi\rangle -\pa_\mu\langle   i (A^\ep)^\mu \pa_\nu\phi , \phi\rangle+\langle   i (A^\ep)^\mu \pa_\nu\phi , \pa_\mu\phi\rangle \\ 
	& =\pa_\nu \langle   i (A^\ep)^\mu\pa_\mu\phi , \phi\rangle -   \langle   i \pa_\nu(A^\ep)^\mu \pa_\mu\phi , \phi\rangle -\pa_\mu\langle   i (A^\ep)^\mu \pa_\nu\phi , \phi\rangle-\langle   i (A^\ep)^\mu \pa_\mu\phi , \pa_\nu\phi\rangle,
\end{align*}
which implies that 
\begin{align}
	\label{eq:phimu}
	&\langle   2i (A^\ep)^\mu\pa_\mu\phi , \pa_\nu\phi\rangle   =\pa_\nu \langle   i (A^\ep)^\mu\pa_\mu\phi , \phi\rangle -   \langle   i \pa_\nu(A^\ep)^\mu \pa_\mu\phi , \phi\rangle -\pa_\mu\langle   i (A^\ep)^\mu \pa_\nu\phi , \phi\rangle.
\end{align}
We therefore can derive that 
\begin{align*}
	\Pi_k(t)-\Pi_k(0) = -\int_0^t \int_{\mathbb{R}^3} \frac{1}{2} \pa_k c|\phi|^2    +   \langle   i \pa_k(A^\ep)^\mu \pa_\mu\phi , \phi\rangle +\pa_t\langle   i (A^\ep)^0 \pa_k\phi , \phi\rangle dx ds.
\end{align*}
For the third term, by using the assumption \eqref{eq:cond4:A} and the estimate \eqref{eq:tildeA:bd} on the connection field $A^\ep$, we can bound that 
\begin{align*}
	&\left|\int_0^t \int_{\mathbb{R}^3}  \pa_t\langle   i (A^\ep)^0 \pa_k\phi , \phi\rangle dx ds\right| \\ 
	&\les  \int_{\mathbb{R}^3}     |(\tilde{A}^\ep)^0(t, x)||\nabla\phi(t, x)| |\phi(t, x)|+\ep^2 |\nabla\phi(0, x)| |\phi(0, x)|  dx \\ 
	&\les \ep^2+\int_{\mathbb{R}^3}\ep^2 \sqrt{1+|x-\xi(t)|} (f_\om(z)(|v|+|\nabla \phi_S|)+|v||\nabla \phi_S| ) +\ep  |v|(|\nabla v|+|v|) dx \\ 
	&\les \ep^2+\ep \|v\|_{H^1}^2. 
\end{align*}
For the first term, recall that $c=-(A^\ep)^\mu (A^\ep)_{\mu}$. We bound that 
\[
|\pa_\mu c|\les |\pa_\mu A^\ep||A^\ep|\les \ep^4 (1+|x_1|+|x_2|)+\ep^2|\tilde{A}^\ep|\les \ep^4 (|x_1|+|x_2|)+\ep^{3}.
\]
By using the weighted mass estimate \eqref{eq:weighted:phi:ex}, we can show that for $t\leq T/\ep$
\begin{equation}
	\label{eq:cphi2}
	\begin{split}
		\left|\int_0^t \int_{\mathbb{R}^3}   \pa_\mu c|\phi|^2    dx ds\right|&\les t \ep^{3}+\ep^4  \int_0^t \int_{\mathbb{R}^3}    (|x_1|+|x_2|)|\phi|^2    dx ds\\ 
		&\les  \ep^2+\ep^4  \int_0^t \int_{|x|\leq s+R_0}  (s+R_0)  |\phi|^2    dx ds+\ep^4  \int_0^t \int_{\Sigma_{s, R_0}}   |x|  |\phi|^2    dx ds\\
		&\les \ep^2+\ep^4 t^2+\ep^4\int_0^t \ep ds \\ 
		&\les \ep^2.
	\end{split}
\end{equation}
Here we may note that $R_0$ relies only on $\la_0$ and we used the fact that 
\begin{align*}
	\int_{\mathbb{R}^3}|\phi|^2dx\les \int_{\mathbb{R}^3}|\phi_S|^2+|v|^2 dx\les 1
\end{align*} 
in view of the bootstrap assumption \eqref{bawv} and the decomposition \eqref{eq:decomp}. Next since
\[
|\pa A^\ep |\les \ep^2,
\]
expand the second term around the   soliton. We can show that 
\begin{align}
	\notag
	& \left|\int_0^t \int_{\mathbb{R}^3}    \langle   i \pa_\nu (A^\ep)^\mu \pa_\mu\phi , \phi\rangle - \langle   i \pa_\nu(A^\ep)^0  \psi_S , \phi_S\rangle-\langle   i \pa_\nu(A^\ep)^k \pa_k\phi_S , \phi_S\rangle dx ds\right|\\ 
	\label{eq:pamuphi}
	& \les \ep^2 \int_0^t \|v\|_{H^1}+\|w\|_{L^2} ds.
\end{align}
Combining the above estimates, we have shown that 
\begin{align*}
	&\left|\Pi_k(t)-\Pi_k(0)+ \int_0^t \int_{\mathbb{R}^3}    \langle   i \pa_k(A^\ep)^\mu \pa_\mu\phi_S , \phi_S\rangle   dx ds\right|\\ 
	& \les \ep^2+\ep^2 \int_0^t \|v\|_{H^1}+\|w\|_{L^2} ds+\ep \|v\|_{H^1}^2
\end{align*}
for $0\leq t\leq T/\ep$. Here $\pa_0\phi_S$ should be understood as $\psi_S$.  

For the standard energy $\Pi_0(t)$, taking derivative and using the  identity \eqref{eq:phimu}, we obtain that 
\begin{align*}
	\pa_t\Pi_0(t)&= \int_{\mathbb{R}^3} \l \pa_t\pa_t\phi, \pa_t\phi \r + \l \pa_t\nabla\phi, \nabla\phi \r+m^2\l \pa_t\phi,\phi \r - |\phi|^{p-1}\l \pa_t\phi, \phi\r dx\\ 
	&= \int_{\mathbb{R}^3} \l \pa_{tt} \phi-\Delta\phi+m^2\phi-|\phi|^{p-1}\phi, \pa_t\phi \r   dx\\
	&=\int_{\mathbb{R}^3} \l  2i (A^\ep)^\mu\pa_\mu\phi+c\phi , \pa_t\phi \r   dx\\ 
	&= \int_{\mathbb{R}^3} \frac{1}{2}\pa_t (c|\phi|^2)-\frac{1}{2}\pa_t c|\phi|^2+ \pa_t \langle   i (A^\ep)^k\pa_k\phi , \phi\rangle -   \langle   i \pa_t (A^\ep)^\mu \pa_\mu\phi , \phi\rangle   dx.
\end{align*}
This indicates that 
\begin{align*}
	\Pi_0(t)-\Pi_0(0)=  \int_0^t\int_{\mathbb{R}^3}   \pa_t (\frac{1}{2}c|\phi|^2+\langle   i (A^\ep)^k\pa_k\phi , \phi\rangle )- \frac{1}{2}\pa_t c|\phi|^2-  \langle   i \pa_t (A^\ep)^\mu \pa_\mu\phi , \phi\rangle   dxds.
\end{align*}
By the assumptions on the connection field $A^\ep$ and the weighted mass bound \eqref{eq:weighted:phi:ex}, we can estimate that 
\begin{align*}
	\int_{\mathbb{R}^3}    \frac{1}{2}|c||\phi|^2   dx  
	&\les \int_{\mathbb{R}^3}    \ep^2(1+\ep^2|x_1|^2+\ep^2 |x_2|^2) |\phi|^2  dx\\ 
	&\les \ep^2+ \int_{ |x|\leq t+R_0}    \ep^4( t+R_0)^2 |\phi|^2dx+ \int_{ \Sigma_{t, R_0}}    \ep^4|x|^2 |\phi|^2dx \\ 
	&\les \ep^2.
\end{align*}
Similarly we can estimate that   
\begin{align*}
	\left|\int_{\mathbb{R}^3}    \langle   i (A^\ep)^k\pa_k\phi , \phi\rangle  dx\right|  
	&\les   \int_{\mathbb{R}^3}    |\tilde{A}^\ep| |\nabla\phi | |\phi| dx+\ep^2|\int_{ |x|\leq t+R_0}    \langle   i  \Omega_{12}\phi , \phi\rangle  dx| +\ep^2\int_{ \Sigma_{t, R_0}}   |x| |  \nabla\phi| | \phi|   dx.  
\end{align*}
Here recall that $\Omega_{12}=x_1\pa_2-x_2\pa_1$. For the first term on the right hand side, we rely on the decomposition of the solution and the estimate \eqref{eq:tildeA:bd} for the connection field $A^\ep$. We can estimate that 
\begin{align*}
	\int_{\mathbb{R}^3}    |\tilde{A}^\ep| |\nabla\phi | |\phi| dx&\les \int_{\mathbb{R}^3}     \ep^2\sqrt{1+|x-\xi(t)|}(|v||\nabla \phi_S|+f_\om(z)(|\nabla \phi_S|+|v|))+ \ep  |v|(|\nabla v|+|v|) dx \\ 
	&\les \ep^2 +\ep \|v\|_{H^1}^2.
\end{align*}
For the second term, using the decomposition \eqref{eq:decomp}, we can further show that 
\begin{align*}
	&\left|\int_{ |x|\leq t+R_0}    \langle   i  \Omega_{12}\phi , \phi\rangle  dx\right| \\ 
	&\les    \left|\int_{ |x|\leq t+R_0}    \langle   i  \Omega_{12}\phi_S , \phi_S\rangle dx\right|+ (t+R_0)\int_{ |x|\leq t+R_0}   (| \nabla\phi_S|+|\phi_S|)(|v|+|\nabla v|)  dx \\ 
	&\les  \int_{ |x|\leq t+R_0} |x_1\pa_2\Theta-x_2\pa_1\Theta | |f_{\omega}(z)|^2     dx+ \ep^{-1} \|v\|_{H^1}\\ 
	&\les \ep^{-1}(|u|+\|v\|_{H^1}).
\end{align*}
Here we used the fact 
\begin{align*}
	\pa_j\Th=-\om (u+u_0)\cdot \pa_j z=-\om (u+u_0)(\pa_j x+(\rho-1)\frac{(u+u_0)\cdot \pa_j x}{|u+u_0|^2}(u+u_0))=-\omega (u+u_0)_j \rho
\end{align*}
and $u_0$ is parallel to $(0, 0, 1)$. The bound for the third term follows from the weighted mass bound \eqref{eq:weighted:phi:ex}. More precisely, we can show that 
\begin{align*}
	\int_{ \Sigma_{t, R_0}}   |x| |  \nabla\phi| | \phi|   dx\leq \left(\int_{ \Sigma_{t, R_0}}   |x|^2  | \phi|^2   dx\right)^{\frac{1}{2}} \|\nabla \phi \|_{L^2}\les 1.
\end{align*}
Here in view of the  bootstrap assumption \eqref{bawv} we have the bound   
\begin{align*}
	&\|\nabla \phi\|_{L^2}\les \|\nabla \phi_S\|_{L^2}+\|v\|_{H^1}\les 1. 
\end{align*} 
We therefore have shown that 
\begin{align*}
	\left|\int_{\mathbb{R}^3}    \langle   i (A^\ep)^k\pa_k\phi , \phi\rangle  dx\right|
	&\les \ep^2+ \ep (|u|+\|v\|_{H^1})+\ep \|v\|_{H^1}^2.
\end{align*}
Now combining all the above estimates and in view of the bounds \eqref{eq:cphi2} and \eqref{eq:pamuphi} for the bulk terms in the expression for $\Pi_0(t)-\Pi_0(0)$, 
we can show  that for all $0\leq t\leq T/\ep$ 
\begin{align*}
	&\left|\Pi_0(t)-\Pi_0(0)+\int_0^t \int_{\mathbb{R}^3}    \langle   i \pa_t(A^\ep)^\mu \pa_\mu\phi_S , \phi_S\rangle dx ds\right| \\ 
	&\les \ep^2+  \ep^2 \int_0^t\|v\|_{H^1}+\|w\|_{L^2}ds+ \ep(|u(t)|+\|v(t, x)\|_{H^1})+\ep \|v\|_{H^1}^2.
\end{align*}
We hence finished the proof for the proposition. 
\end{proof}
\subsubsection{Improving the Bootstrap Assumptions on the Remainder Terms}\label{sec:X:zth}
To derive energy for the remainder terms $(v, w)$, consider the combination 
\[
\Pi_0(t)+(u_k(0)+(u_0)_k)\cdot \Pi_k(t)-\frac{\om}{\rho(0)}Q(t).
\]
In view of the  decompositions \eqref{Pidecomp}, \eqref{Qdecomp}, \eqref{Hdecomp}, the soliton part is 
\[
\rho G (1-(u(0)+u_0)\cdot (u+u_0))-\frac{\om^2}{\rho(0)}\|f_\om\|_{L^2(\mathbb{R}^3)}^2.
\]
For the quadratic part in terms of the remainder terms $(v, w)$, we compute that 
\begin{align*}
	&    \f12\int_{\mathbb{R}^3}m^2 |v|^2+|\nabla (\e^{i\Th}v)|^2 +|w|^2-f_\om^{p-1}|v|^2 -(p-1)f_\om^{p-1}v_1^2+2\frac{\om}{\rho}\langle iw, v\rangle \\
	&\quad +2\langle\e^{i\Th}w, \pa_k(\e^{i\Th}v)\rangle(u +u_0 )_k dx\\ 
	= &\frac{1}{2\rho}(\|w+\rho (u+u_0)\nabla_z v-i\rho \om v\|_{L^2(dz)}^2+\langle v_1, L_{+}v_1\rangle_{dz}+\langle v_2, L_{-}v_2\rangle_{dz}) \\ 
	=&\frac{1}{2\rho}E(v, w, \la),
\end{align*}
which can be viewed as the energy for $(v, w)$. Here the linearized operators   $L_{+}$, $L_{-}$ are defined in \eqref{L+}. In view of Proposition \ref{positen}, we see that under the orthogonality condition \eqref{orthcond}, $E(v, w, \la)$ is equivalent to $\|v\|_{H^1}^2+\|w\|_{L^2}^2$. We hence can write that
\begin{equation*}
	\begin{split}
		&\Pi_0(t)+(u^k(0)+u_0^k)\cdot \Pi_k(t)-\frac{\om}{\rho(0)}Q(t)\\ 
		=&\rho G(1-(u(0)+u_0)\cdot
		(u+u_0))-\frac{\om^2}{\rho(0)}\|f_\om\|_{L^2}^2+\frac{1}{2\rho}E(v, w, \la)+O(  \|v\|_{H^1}^3 ). 
	\end{split}
\end{equation*}
Denote
\begin{equation}
	\label{defdeltaH}
	\begin{split}
		d\mathcal{H}(t)&=\rho G(1-(u(0)+u_0)\cdot (u+u_0))
		-\frac{\om^2}{\rho(0)}\|f_\om\|_{L^2}^2-\frac{G(0)}{\rho(0)}
		+\frac{\om(0)\om}{\rho(0)}\|f_{\om}\|_{L^2}^2(0)
	\end{split}
\end{equation}
and
\begin{equation*}
	\begin{split}
		Err_0=
		\Pi_0(0)+(u^{k}(0)+u_0^k)\Pi_k(0)-\frac{G(0)}{\rho(0)}-\frac{\om}{\rho(0)}\left(Q(0)-\om(0)\|f_{\om}\|_{L^2}^2(0)\right).
	\end{split}
\end{equation*}
We then have 
\begin{equation}
	\label{energycomb1}
	\begin{split}
		\frac{1}{2\rho}E(v, w, \la)+d \mathcal{H}(t) =&\Pi_0(t)-\Pi_0(0)+(u^k(0)+u_0^k)\cdot (\Pi_k(t)-\Pi_k(0))\\
		&-\frac{\om}{\rho(0)}(Q(t)-Q(0))+O( \|v\|_{H^1}^3 )+Err_0.
	\end{split}
\end{equation}
By definition, initially we have 
\[
\rho(0)G(0)-|u(0)+u_0|^2\rho(0)G(0)-\frac{G(0)}{\rho(0)}=0.
\]
In view of Lemma \ref{datapre}, the initial data verify the bound 
\begin{equation}
	\label{err3est}
	|Err_0|\les \ep^2.
\end{equation}
Now we show that $d \mathcal{H}(t)$ is also non-negative. 
\begin{lemma}
	\label{lempositdH}
	Let $\la(t)\in \La_{\delta_0}$. If $\delta_0$ is sufficiently small, depending only $\la(0)$, then
	\[
	d \mathcal{H}(t)\geq c_3(|\om(t)-\om(0)|^2+|u(t)-u(0)|^2)
	\]
	for some positive constant $c_3$ depending only on $\la(0)$.
\end{lemma}
\begin{proof}
	Recall that 
	\[
	\rho=(1-|u+u_0|^2)^{-\frac{1}{2}}. 
	\]
	As functions of the variable $u$, we  compute that 
	\begin{align*}
		\pa_u (\rho(1-(u(0)+u_0)\cdot (u+u_0)))(0)=&\pa_u\rho(0)\rho(0)^{-2}-\rho(0) (u(0)+u_0)=0,\\
		\pa^2_u (\rho(1-(u(0)+u_0)\cdot (u+u_0)))(0)=&\pa_u^2\rho(0)
		(1-|u(0)+u_0|^2)-2 \pa_u\rho(0) \cdot(
		u(0)+u_0)\\
		=&\rho(0)I+\rho(0)^{3}(u(0)+u_0)\cdot (u(0)+u_0).
	\end{align*}
	In particular the Hessian is positive definite. In view of the bootstrap assumption \eqref{bawv}, for sufficiently small $\delta_0$, depending only on $\la(0)$, there exists a constant $c_1$, depending only on $u(0)+u_0$, such that 
	\[
	\rho(1-(u(0)+u_0)\cdot (u+u_0))\geq \rho(0)^{-1}+c_1|u(t)-u(0)|^2,\quad |u(t)-u(0)|\leq \delta_0. 
	\]
	On the other hand, recall the definition of $G(\omega)$ in \eqref{defofB}, which relies only on $\omega$. In view of the  scaling property of the ground state in Proposition \ref{prop:ground:f}, we can show that
	\begin{align*}
		& \frac{d}{d\omega}\left(G(t)-\om^2\|f_\om\|_{L^2}^2-G(0)+\om(0)\om\|f_\om\|_{L^2}^2(0)\right)|_{t=0}=0,\\
		& \frac{d^2}{d\omega^2}\left(G(t)-\om^2\|f_\om\|_{L^2}^2-G(0)+\om(0)\om\|f_\om\|_{L^2}^2(0)\right)|_{t=0}\\
		&\quad=\left(\frac{6-2p}{p-1}\om(0)^2-m^2\right)(m^2-\om(0)^2)^{\frac{9-5p}{2(p-1)}}\|f\|_{L^2}^2, 
	\end{align*}
	which is positive due to the assumption $\la(0)\in \La_{\delta_0}$. Here $G(t)=G(\omega(t))$. Therefore  if $\delta_0$ is sufficiently small, we have
	\[
	G(t)-\om^2\|f_\om\|_{L^2}^2-G(0)+\om(0)\om\|f_\om\|_{L^2}^2(0)\geq c_2|\om(t) -\om(0)|^2,\quad \forall |\om(t)-\om_0| <\delta_0 
	\]
	for some constant $c_2>0$ depending only on $\la(0)$. The above convexity then implies that  
	\begin{align*}
		d\mathcal{H}(t)&\geq \rho(0)^{-1}G(t)+c_1|u(t)-u(0)|^2G(t)-\rho(0)^{-1}(\om^2\|f_\om\|_{L^2}^2+G(0)-\om(0)\om\|f_{\om}\|_{L^2}^2(0))\\
		&\geq c_1|u(t)-u(0)|^2G(t)+c_2\rho(0)^{-1}|\om(t)-\om(0)|^2\geq c_3(|\om(t)-\om(0)|^2+|u(t)-u(0)|^2)
	\end{align*}
	for some constant $c_3$ depending only on $\la(0)$ if $\delta_0$ is sufficiently small.
\end{proof}
Since $(v, w)$ satisfies the orthogonality condition \eqref{orthcond}, Proposition \ref{positen} implies that $E(v, w, \la)$ is equivalent to the energy $\|v\|_{H^1(\mathbb{R}^3)}^2+\|w\|_{L^2(\mathbb{R}^3)}^2$ when $\ep$ is sufficiently small. The above lemma then implies that the left hand side of  \eqref{energycomb1} is positive and has a lower bound 
\[
\frac{1}{2\rho}E(v, w, \la)+d\mathcal{H}\geq c_4(\|v\|_{H^1}(t)+\|w\|_{L^2}(t)+|\om(t)-\om(0)|^2+|u(t)-u(0)|^2)
\]
for some constant $c_4$ depending only on $\la(0)$ for sufficiently small  $\ep>0$. Now for the right hand side of the energy identity \eqref{energycomb1}, Proposition \ref{propcon} shows that the energies $\Pi_\mu(t)$ is not almost conserved due to the fact that the soliton part is of order $\ep^2 t$. The key observation is that the soliton travels along the direction $(1, u_0)$ in the Minkowski space $\mathbb{R}^{1+3}$ and hence $(\pa_t+u_0\nabla)\phi_S$ is of order $\ep$ instead of $1$. Let $X$ be the vector field 
\[
X=\pa_t+u_0^k\pa_k
\]
in the Minkowski space $\mathbb{R}^{1+3}$. Integration by parts, we note that 
\begin{align*}
	& |\int_0^t \int_{\mathbb{R}^3}    \langle   i \pa_t(A^\ep)^\mu \pa_\mu\phi_S , \phi_S\rangle dx ds+ u_0^k \int_0^t \int_{\mathbb{R}^3}    \langle   i \pa_k(A^\ep)^\mu \pa_\mu\phi_S , \phi_S\rangle   dx ds| \\
	& \les |  \int_{\mathbb{R}^3}    \langle   i (A^\ep)^\mu \pa_\mu\phi_S , \phi_S\rangle dx|_{0}^t  |+\int_0^t \int_{\mathbb{R}^3}     | (A^\ep)^\mu| | X\l i\pa_\mu\phi_S  , \phi_S\r  | dx ds. 
\end{align*}
By using the assumption \eqref{eq:cond4:A} on the connection field $A^\ep$ and the bound \eqref{eq:Omphis}, we can show that 
\begin{align*}
	|\int_{\mathbb{R}^3}    \langle   i (A^\ep)^\mu \pa_\mu\phi_S , \phi_S\rangle dx| &\les \ep^2  \int_{\mathbb{R}^3}      \sqrt{1+|x-\xi(t)|}(|\psi_S|+|\nabla \phi_S|) | \phi_S|  + |\l i\Omega_{12}\phi_S, \phi_S \r |dx \\ 
	&\les \ep^2 + \ep^2\int_{\mathbb{R}^3}        |x_1 \pa_2\Theta-x_2\pa_1\Theta| |f_{\omega}(z)|^2dx \\ 
	& \les \ep^2 + \ep^2\int_{\mathbb{R}^3}         (|z|+|\eta|)|u| |f_{\omega}(z)|^2dx \\
	&\les \ep^2+\ep^2 C_2 \delta_1 \les \ep^2.
\end{align*}
The last step follows from the bootstrap assumptions \eqref{baxi}, \eqref{bawv} and the assumption $C_2\delta_1\leq 1$. Next we compute that 
\begin{align*}
	\pa_t z &= -\dot \xi -(\rho-1)P_{u+u_0} \dot \xi +\dot{u}\pa_u( (\rho-1)P_{u+u_0}) (x-\xi)\\ 
	&=-\dot \eta -u_0 -(\rho-1)P_{u+u_0}  ( \dot \eta+u_0) +\dot{u}\pa_u( (\rho-1)P_{u+u_0}) (x-\xi),\\ 
	\pa_k z&=\pa_k x+(\rho-1)P_{u+u_0}\pa_k x.  
\end{align*}
In particular for $\la\in \Lambda_{\delta_0}$ and in view of Corollary \ref{corcontrga}, we have 
\begin{align*}
	|X z|=|(\pa_t+ u_0^k\pa_k)( z)|\les |\dot \eta|+|\dot{u}||z|\les |\dot{\gamma}|(1+|z|)+|u| \les |u|+(1+|z|)(C_2\ep^2+\|v\|_{H^1}^2).
\end{align*}
Now we can compute that 
\begin{align*}
	&\l i \psi_S, \phi_S\r= \l i(  i\rho\om
	f_{\om}  -\rho (u+u_0)\cdot \nabla_z f_\om ) , f_\om \r =-\rho\om f_\om^2, \\ 
	&\l i\pa_k\phi_S, \phi_S\r =-\pa_k\Th f_\om^2=\rho (u+u_0)_k  f_\om^2.
\end{align*}
By using the bound \eqref{eq:Omphis}, we therefore can estimate that 
\begin{align*}
	&\int_0^t \int_{\mathbb{R}^3}     | (A^\ep)^\mu| |X \l i\pa_\mu\phi_S, \phi_S\r  |  dx ds \\ 
	&\les \ep^2 \int_0^t \int_{\mathbb{R}^3}     \sqrt{1+|z|} (|X(\rho f_\om^2)|+ (|X\om|+|Xu|)f_\om^2 ) dx ds \\ 
	&\quad +\ep^2 \int_0^t \int_{\mathbb{R}^3}     |x_1 X(\rho (u+u_0)_2 f_\om^2)-x_2 X(\rho(u+u_0)_1f_\om^2) |dx ds \\ 
	&\les \ep^2 \int_0^t C_2\ep^2+\|v\|_{H^1}^2+|u|  dt   +\ep^2 \int_0^t C_2|\dot u|+   (C_2\ep^2+\|v\|_{H^1}^2+|u|+|\dot u|)C_2|u|  dt   \\ 
	&\les \ep^3 C_2^2+ \ep^2 C_2^2\int_0^t \|v\|_{H^1}^2  dt+\ep^2 \int_0^t |u|dt\\ 
	&\les \ep^2+\ep \int_0^t \|v\|_{H^1}^2+\ep |u| dt. 
\end{align*}
Here we used the estimate \eqref{eq:tildeA:bd} for the connection field $\tilde{A}$ and the bootstrap assumption \eqref{baxi}. We also note that the constant $C_2$ verifies the bound  $C_2^4 \ep\leq 1$ and the vector field $u_0$ is parallel to $(0, 0, 1)$. In particular we have shown that 
\begin{align*}
	& |\int_0^t \int_{\mathbb{R}^3}    \langle   i \pa_t(A^\ep)^\mu \pa_\mu\phi_S , \phi_S\rangle dx ds+(u^k(0)+u_0^k)\int_0^t \int_{\mathbb{R}^3}    \langle   i \pa_k(A^\ep)^\mu \pa_\mu\phi_S , \phi_S\rangle   dx ds| \\
	& \les  \ep^2+\ep \int_0^t \|v\|_{H^1}^2 +\ep |u|dt. 
\end{align*}
Here we may note that initially $|u(0)|\les \ep$. Thus in view of Proposition \ref{propcon}, for $\la\in \La_{\delta_0}$, we can bound that 
\begin{align*}
	&|\Pi_0(t)-\Pi_0(0)+(u^k(0)+u_0^k)\cdot (\Pi_k(t)-\Pi_k(0)) -\frac{\om}{\rho(0)}(Q(t)-Q(0))|\\ 
	&\les \ep^2+\ep \int_0^t \|v\|_{H^1}^2 dt+  \ep^2 \int_0^t\|v\|_{H^1}+\|w\|_{L^2}+|u|ds  + \ep(|u(t)|+\|v(t, x)\|_{H^1})+\ep^{\frac{3}{2}}\|v\|_{L^2}^2.
\end{align*}
The energy identity \eqref{energycomb1} together with the bound \eqref{err3est} and Proposition \ref{positen}, Lemma \ref{lempositdH} leads to the energy estimate 
\begin{align*}
	&|u(t)-u(0)|^2+|\omega(t)-\omega(0)|^2+\|v\|_{H^1}^2+\|w\|_{L^2}^2 \\ 
	&\leq C\left(  \ep^2+(\ep^{\frac{3}{2}}+\delta_1)\|v\|_{H^1}^2+ \ep \int_0^t \|v\|_{H^1}^2 dt+ \ep^2 \int_0^t\|v\|_{H^1}+\|w\|_{L^2}+|u|ds  + \ep(|u |+\|v \|_{H^1}) \right) 
\end{align*}
for some constant $C$ depending only on $p$, $\la_0$, $m$ and $\delta_0$. Chose $\delta_1$ such that 
\[
\delta_1=\frac{1}{4C}.
\]
For sufficiently small $\ep$, the terms $\ep^{\frac{3}{2}}\|v\|_{H^1}^2$ and $O(\|v\|_{H^1}^3)$ on the right hand side can be absorbed. Notice that 
\[
|u(0)+u_0-u_0|\les \ep
\]
in view of Lemma \ref{datapre}. A simple bootstrap argument and Gronwall's inequality then imply that 
\begin{align*}
	&|u(t)-u(0)|+|\omega(t)-\omega(0)|+\|v\|_{H^1}+\|w\|_{L^2}  \leq C_3 \ep
\end{align*}
for some constant $C_3$ depending only on $\la_0$, $m$, $p$, $T$ and $\delta_0$.  
If 
$$ \ep<\min\{\delta_1 C_3^{-1}=\frac{1}{4CC_3},\quad \frac{1}{2}\delta_0 C_3^{-1} \},$$
we then can improve the bootstrap assumption \eqref{bawv}. From Corollary \ref{corcontrga}, we conclude that 
\[
|\dot{\gamma}|\les C_2\ep^2+\ep^2\les C_2\ep^2. 
\]
By definition, recall that 
\[
\xi(t)=\eta(t)+u_0t,\quad |\dot{\xi}(t)-u(t)-u_0|\les |\dot{\gamma}|\les C_2\ep^2.
\]
Therefore for $t\leq T/\ep$, we have 
\begin{align*}
	|\eta(t)|\les |\eta(0)|+\int_0^t |\dot{\eta}(s)|ds\les 1 +\int_0^t C_2\ep^2 +|u(s)|ds\les 1+C_2\ep^2 t+\ep t\les 1+C_2\ep\leq C_2 
\end{align*}
for sufficiently small $\ep$, depending only on $\la_0$, $m$, $p$, $T$ and $\delta_0$. This in particular improves the bootstrap assumption \eqref{baxi}. We remark here that the constant $C_2$ can be chosen to be $C_2=\delta_1^{-1}$. We hence finished the proof for Proposition \ref{prop:fix}.
\subsection{Higher order energy estimates for the scalar field}
This section is devoted to the proof for Proposition \ref{pro:phi:HSob}. We have shown in the previous section that stable soliton on a constant small electric magnetic field is orbital stable up to time $T/\ep$. To solve the full Maxwell-Klein-Gordon system \eqref{eq:MKG:scaled} and  improve the bootstrap assumption \eqref{eq:cond4:A} on the connection field $\tilde{A}$, we need higher order energy estimates for the scalar field $\phi$. The problem is that to derive $H^3$ estimates for the connection field $\tilde{A}^\ep$ under Lorentz gauge, we need to commute the equation with two derivatives. The key observation that allows us to prove the main theorem under the weak condition that $\delta$ is sufficiently small which is independent of $\ep$ is that the soliton travels along the timelike vector field $X$. Integration by parts gives an extra order of smallness.  This means that we may need to take fourth order derivatives on the soliton or on the modulation curve $\gamma(t)$, which, however, is not quite possible as the nonlinearity is of the form $|\phi|^{p-1}\phi$ with $p<\frac{7}{3}$. The idea, as in \cite{yang4}, is to choose a modified curve $\tilde{\la}(t)\in \La_{\delta_0}$, defined as the integral curve of $V(\la)$, that is
\begin{equation*}
	\pa_t\tilde{\la}=V(\la),\quad
	\tilde{\la}(0)=(\om_0, \th_0, 0, u(0)+u_0).
\end{equation*}
With this modified curve $\tilde{\la}(t)$,  the solution $\phi$ has the following new decomposition
\begin{equation}
	\label{eq:decomp:new} 
	\phi(t, x)=\phi_S(\tilde{\la};x)+\e^{i\Th(\tilde{\la})}\tilde{v}.
\end{equation}
We show that under this new decomposition, the remainder term $\tilde{v}$ is still of order $\ep$ in $H^1$.
\begin{lemma}
	\label{lem:tildev}
	It holds that 
	\[
	\|D\tilde{v}(t, x)\|_{L_x^2}+\| \tilde{v}(t, x)\|_{H_x^1}\les \ep,\quad \forall t\in[0, T/\ep].
	\]
\end{lemma}
\begin{proof}
	Note that 
	\begin{align*}
		\|v\|_{H^1}+\|w\|_{L^2}\les \ep,\quad |\dot{\gamma}|=|\pa_t\lambda-\pa_t\tilde{\lambda}|\les \ep^2. 
	\end{align*}
	By definition, we can estimate that  
	\begin{align*}
		\|\tilde{v}\|_{H^1}&=\|\e^{-i\Th(\tilde{\la})}(\phi_S(\la;x)+\e^{i\Th(\la)} v-\phi_S(\tilde{\la};x))\|_{H^1} \\
		&\les \|  \e^{i\Th(\la)-i\Th(\tilde{\la})} v \|_{H^1}+\|\e^{-i\Th(\tilde{\la})}(\phi_S(\la;x) -\phi_S(\tilde{\la};x))\|_{H^1}\\
		&\les \|v\|_{H^1}+|\la-\tilde{\la}|\les \ep. 
	\end{align*}
	Similarly for the time derivative of $\tilde{v}$, we have 
	\begin{align*}
		\|\e^{i\Th(\tilde{\la})}\pa_t
		\tilde{v}\|_{L^2}&=\|\pa_t(\phi-\phi_S(\tilde{\la};x))-i\pa_t\Th(\tilde{\la})\e^{i\Th(\tilde{\la})}\tilde{v}\|_{L^2}\\
		&=\|\pa_\la\phi_S(\la;x)\cdot V(\la)+\e^{i\Th(\la)} w-\pa_\la\phi_S(\tilde{\la};x)\cdot V(\la)-i\pa_t\Th(\tilde{\la})\e^{i\Th(\tilde{\la})}\tilde{v}\|_{L^2}\\
		&\les |\la-\tilde{\la}| +\|w\|_{L^2}+\| \pa_t\Th(\tilde{\la}) \tilde{v}\|_{L^2}\\
		&\les \ep
	\end{align*}
	for all $0\leq T/\ep$. Now since 
	\[
	|D\tilde{v}|\leq |\pa\tilde{v}|+|A^\ep \tilde{v}|,
	\]
	the lemma then follows if we can bound $\|A^\ep \tilde{v}\|_{L^2}$. In view of the estimate \eqref{eq:tildeA:bd} for the connection field $\tilde{A}^\ep$ and the weighted energy estimate \eqref{eq:weighted:phi:ex} in the exterior region, we can show that 
	\begin{align*}
		\int_{\mathbb{R}^3} |A^\ep \tilde{v}|^2 dx &\leq \int_{ |x|\leq t+R_0} |\tilde{v}|^2(\ep^2+\ep^4(x_1^2+x_2^2))dx+\int_{\Sigma_{t, R_0}} (\ep^2+\ep^4(x_1^2+x_2^2)) |\phi-\phi_S(x;\tilde{\la})|^2dx \\ 
		&\les (\ep^4 (t+R_0)^2+\ep^2)\|\tilde{v}\|_{L^2}^2+\ep^4+\int_{\Sigma_{t, R_0}} (\ep^2+\ep^4( 1+|z|^2))  |f_{\om}(z) |^2 dx\\ 
		&\les \ep^2.
	\end{align*}
	Here we used the fact that $$|x_1|+|x_2|\leq |\eta|+|z|\les 1+|z|.$$ See the bootstrap assumption \eqref{baxi} on the center of the soliton. We hence finished the proof for the Lemma. 
\end{proof}
\subsubsection{Estimates for Higher Derivatives of $\la(t)$}
Higher order energy estimates of the scalar field involves higher order derivatives of the modulation curve. We have shown that the time derivative of the modulation curve $\gamma$ stays small of order $\ep^2$ up to time $T/\ep$. We first show that the higher order derivatives of the modulation curve are also small. 
\begin{lemma}
	\label{lem:Y:N} 
	Let $\mathcal{N}(\la)$ be defined in line \eqref{eq:def4:N}. Assume $p\geq 2$. Then for any vector field $Y_1$, $Y_2$ in the Minkowski space $\mathbb{R}^{1+3}$, it holds that 
	\begin{align*}
		|D_{Y_1}\mathcal{N}(\la)|\les &(|v|+|v|^{p-1})(|D_{Y_1}
		f_\om|+|D_{Y_1} v|+|v|),\\
		|D_{Y_1}D_{Y_2}\mathcal{N}(\la)|\les &  (1+\|v\|_{L^\infty}+\sum\limits_{j=1}^2|D_{Y_j} \ln
		f_\om| )(\sum\limits_{j=1}^2|D_{Y_j}  f_\om|^2 +|D_{Y_j}
		v|^2 +|v|^2) \\ 
		&+(|v|+|v|^{p-1})(|D_{Y_1}D_{Y_2} f_\om|+|D_{Y_1}D_{Y_2}v|).
	\end{align*}
	Here the covariant derivative $D$ is with respect to the connection field $A^\ep$. 
\end{lemma}
\begin{proof}
	It follows by direct calculations and the properties of $f_\om$ summarized in Theorem \ref{prop:ground:f}.
\end{proof}
With this lemma and the modulation equations \eqref{modeq}, we are now able to estimate the higher order derivatives of the modulation curve $\la(t)$.
\begin{proposition}
	\label{prop:higher:mod}
	The modulation curve $\gamma(t)$ verifies the following higher order estimates 
	\begin{align*}
		&|\ddot{\ga}|\les\ep^2,\\
		&|\pa_t^3\ga|\les\ep^2(1+\|v\|_{L^\infty})+\ep\|\pa_\la\phi_S X^2v\|_{L^2}
	\end{align*}
	for all $0\leq t\leq T/\ep$. Here $X=\pa_t+ u_0^k\pa_k$ is the timelike vector field. 
\end{proposition}
\begin{proof}
	In view of the modulation equations \eqref{modeq}, we obtain the ODE for $\dot{\ga}$
	\begin{equation*}
		(M_0+M_1)\ddot{\ga}+\pa_t(M_0+M_1)\dot{\ga}=\pa_t K(t;\la(t)).
	\end{equation*}
	From the proof of Lemma \ref{nondegD}, the components of the  matrix $M_0$ relies only on $\omega$ and $u$. Since we have shown in the previous section that 
	\[
	|\dot{\omega}|+|\dot{u}|\leq |\dot\gamma|\les \ep^2, \quad |\dot\la|=|\dot\gamma+V(\la)|\les 1, 
	\]
	we see that 
	\[
	|\pa_t M_0|\les |\dot{\omega}|+|\dot{u}|\les \ep^2. 
	\]
	By the definition of  $M_1$   in  \eqref{defofM}, we rely on the equation \eqref{defofH} for the scalar field $\phi$.   We show that
	\begin{align*}
		|\pa_t M_1|&\les |\dot \la|\cdot \|w\|_{L^2}+|\langle \pa_\la^2\phi_S, \pa_t(\e^{i\Th}w)\rangle_{dx}|+
		|\dot \la|\cdot \|v\|_{L^2}+|\langle\pa_{\la}^2\phi_S, \pa_t\left(\phi-\phi_S\right)\rangle_{dx}|\\
		& \les \ep +|\langle\pa_{\la}^2\phi_S,
		\phi_{tt}-\pa_t\psi_S \rangle_{dx}|+|\langle\pa_{\la}^2\psi_S,  \psi_S+e^{i\Th}w-\pa_{\la}\phi_S (V(\la)+\dot \ga)\rangle_{dx}|\\
		&\les \ep +|\langle\pa_\la^2\phi_S,  \Delta \phi-(m^2\phi-|\phi|^{p-1}\phi)+b^\mu\pa_\mu\phi+c\phi-\pa_t\psi_S\rangle_{dx}|
		+\|w\|_{L^2}+|\dot{\ga}|\\
		&\les \ep + |\langle\pa_\la^2\phi_S,  \Delta (e^{i\Th}v)-(m^2 e^{i\Th}v-|\phi|^{p-1}\phi+f_{\om}^{p-1}\phi_S)+b^\mu\pa_\mu\phi+c\phi \rangle_{dx}|\\
		&\les \ep + \|\nabla (e^{i\Th}v)\|_{L^2} \|\nabla \pa_\la^2\phi_S\|_{L^2}+\| v\|_{L^2}  +\ep \|\pa\phi\|_{L^2}+\ep^2 \|\phi\|_{L^2}\|(|x_1|^2+|x_2|^2)\pa_\la^2\phi_S\|_{L^2}\\ 
		&\les \ep +\|v\|_{H^1}+\ep^2\|(1+|z|^2)\pa_\la^2\phi_S\|_{L^2}\\ 
		&\les \ep. 
	\end{align*}
	Here we have used the identity \eqref{idenofphiS} for the translated solitons. We hence have shown
	\[
	|\pa_t(M_0+M_1)\dot\gamma|\les \ep^3,\quad \forall t\leq T/\ep.
	\]
	Since $M_0+M_1$ is non-degenerate in view of  Lemma \ref{nondegD}, it remains to control the time derivative of the nonlinearity $K(t;\la(t))$. By the definition of $K(t;\la(t))$ in \eqref{eq:def4F}, when the derivative hits on the soliton part, the estimate is the same as the proof of Lemma  \ref{lemDF}. More precisely the proof of Lemma \ref{lemDF} also implies that 
	\begin{align*}
		| \langle \pa_t\pa_\la \phi_S, b^\mu \pa_\mu\phi+c\phi +\e^{i\Th}\mathcal{N}(\la)\rangle_{dx} |\les | \langle \pa^2_\la \phi_S, b^\mu \pa_\mu\phi+c\phi +\e^{i\Th}\mathcal{N}(\la)\rangle_{dx} |\les \ep^2. 
	\end{align*}
	For the other terms when the time derivative hits on the connection field, in view of the assumption \eqref{eq:cond4:A} and the estimate \eqref{eq:tildeA:bd}, we can show that 
	\begin{align*}
		& |\langle  \pa_\la \phi_S, \pa_t(b^\mu\pa_\mu\phi+c\phi)\rangle_{dx}| \\ 
		&\les \ep^2 \| (1+|x_1|^2+|x_1|^2) \pa_\la \phi_S\|_{L^2} (\|\pa_\mu\phi\|_{L^2}+\|\phi\|_{L^2})  +|\langle  \pa_k(b^k\pa_\la \phi_S),  \pa_t\phi \rangle_{dx}|+|\langle   \pa_\la \phi_S, b^0 \pa_{tt}\phi \rangle_{dx}|\\ 
		&\les \ep^2+\ep^2\|\sqrt{1+|z|}\pa_k\pa_\la\phi_S\|_{L^2}+|\langle   b^0\pa_\la \phi_S,   \Delta \phi-(m^2\phi-|\phi|^{p-1}\phi)+b^\mu\pa_\mu\phi+c\phi \rangle_{dx}|\\ 
		&\les  \ep^2 +\|\nabla (b^0\pa_\la \phi_S)\|_{L^2} \|   \nabla \phi\|_{L^2} +\|(|b^\mu|+|c|+1) b^0\pa_\la \phi_S\|_{L^2}(\|\phi\|_{L^2}+\|\pa\phi\|_{L^2}+\||\phi|^p\|_{L^2})  \\ 
		&\les \ep^2+\ep^2\|\pa_\la \phi_S\|_{L^2}+\ep^2\|( |x_1|^2+|x_2|^2+1) \sqrt{1+|x-\xi(t)|} \pa_\la \phi_S\|_{L^2}\\ 
		&\les \ep^2
	\end{align*}
	Here we used the equation \eqref{defofH} for the scalar field $\phi$ and the bound 
	\[
	|x-\xi(t)|\les |z|,\quad |x_1|+|x_2|\les |\eta(t)|+|z|\les 1+|z|. 
	\]
	The most difficult term turns out to be the one when the time derivative hits on the nonlinearity $\mathcal{N}$ since in view of Lemma \ref{lem:Y:N} it may not be of order $\ep^2$. The key observation is that the soliton travels along the timelike vector field $X$. Recall that we have shown in section \ref{sec:X:zth}
	that 
	\[
	|Xz| \les |u|+(1+|z|)(\ep^2+\|v\|_{H^1}^2)\les \ep+  \ep^2(1+|z|). 
	\]
	This in particular implies that 
	\begin{align*}
		|X f_\om(z)|=|\pa_\om f_\om \dot{\om}+\nabla_z f_\om(z) Xz|\les \ep.
	\end{align*}
	Here we may note that the ground state $f_\om(z)$ decays exponentially in $z$. Now by using  Lemma \ref{lemnonlinear} and Lemma \ref{lem:Y:N}, we can estimate that 
	\begin{align*}
		&|\langle\pa_\la\phi_S, \pa_t\mathcal{N}(\la)\rangle_{dx}| \\ 
		&\les
		|\langle\pa_\la\phi_S, X
		\mathcal{N}(\la)\rangle_{dx}|+|\langle\pa_\la\phi_S, u_0\nabla_x\mathcal{N}(\la)\rangle_{dx}|\\
		&\les\|\pa_\la\phi_S(|v|+|v|^{p-1})(|X f_\om|+|X
		v|+|v|)\|_{L^1}+|\langle u_0\nabla_x(\pa_\la\phi_S), \mathcal{N}(\la)\rangle_{dx}|\\
		&\les(\|\pa_\la\phi_S \cdot X v\|_{L^2} +\ep)(\|v\|_{L^2}+\|v\|_{H^1}^{p-1})+\|v\|_{H^1}^2+\|v\|_{H^1}^p\\
		&\les \ep^2+\ep \|\pa_\la\phi_S \cdot \pa_t v\|_{L^2}+ \ep \|\nabla v\|_{L^2}\\ 
		&\les
		\ep^2+\ep \| \pa_\la\phi_S \cdot\pa_t(\e^{-i\Th}(\phi-\phi_S))\|_{L^2}  \\
		&\les\ep^2+\ep(\|w\|_{L^2}+\|v\|_{H^1})\\ 
		&\les \ep^2.
	\end{align*}
	We remark here  although $\dot\Th$ depends on $z$, the soliton $|\pa_\la\phi_S|$ decays exponentially in $z$, which in particular shows that $|z\pa_\la\phi_S|$ is uniformly bounded. The above computations then indicate that 
	\[
	|\pa_t K(t;\la(t))|\les \ep^2.
	\]
	Therefore the equation for $\ddot{\ga}$ together with the nondegeneracy of the leading matrix $M_0+M_1$ proven in Lemma \ref{nondegD} implies that 
	\[
	|\ddot{\ga}|\les \ep^2. 
	\]
	We proceed in a similar manner to estimate the third derivative of the modulation curve. Taking time derivative again on the modulation equation  \eqref{modeq}, we derive that 
	\[
	(M_0+M_1 )\pa_{t}^3
	\ga+2\pa_t(M_0+M_1 )\ddot{\ga}+\pa_{tt}(M_0+M_1)\dot{\ga}=\pa_{tt}K(t;\la(t)).
	\]
	We have already shown above that $\pa_t(M_0+M_1)$ is uniformly bounded. For the second time derivative of $M_0+M_1$, note that the components of $M_0$ relies only on $u$ and $\om$, which in particular implies that 
	\[
	|\pa_{tt}M_0|\les |\ddot{\ga}|+|\dot{\gamma}|\les \ep^2. 
	\]
	By the definition of the matrix $M_1$, the components are the inner product of a soliton ($\pa_\la\phi_S$ or $\pa_\la\psi_S$) and the remainder terms $v$ or $w$. 
	When the time derivative hits on the soliton part, the bound is similar to that of $\pa_t M_1$. The main new terms in $\pa_{tt}M_1$ are those when all the time derivatives hit on the remainder terms since we do not have higher order energy estimates for the solutions. The idea, which has already been used above, is to replace $\pa_{tt}\phi$ in view of the equation \eqref{defofH}. The higher order spatial derivatives of the solution can be transfered to the soliton part through integration by parts. This procedure leads to the same bound 
	\[
	|\pa_{tt}M_1|\les \ep.
	\]
	For the nonlinear term $K(t;\la(t))$, like the case $\pa_t K$, we only discuss those new terms when all the time derivatives hit on the connection field $\tilde{A}$ or the nonlinearity $\mathcal{N}$. By using the assumption \eqref{eq:cond4:A} and the bound \eqref{eq:tildeA:bd} for the connection field $\tilde{A}$, we can estimate that 
	\begin{align*}
		|\langle \pa_{tt}b^\mu \pa_\la \phi_S, \pa_\mu\phi\rangle_{dx}|+|\langle \pa_\la\phi_S, \pa_{tt}c \phi\rangle_{dx}|\les \|\pa^2 \tilde{A}^\ep \|_{L^2}\| \pa\phi\|_{L^2} +\|\pa^2\tilde{A}^\ep\|_{L^2} \|\phi \pa_\la\phi_S (1+ |z|)\|_{L^2}\les \ep^2.
	\end{align*}
	The main term is still the one when all the time derivatives hit  on the nonlinearity. By using Lemma \ref{lem:Y:N} and the above bound for $Xf_\om(z)$, we can show that 
	\begin{align*}
		&\left|\langle\e^{-i\Th}\pa_\la\phi_S, \pa_{tt}\mathcal{N}(\la)\rangle_{dx}\right| \\
		&\les\left|\langle\e^{-i\Th}\pa_\la\phi_S, X^2\mathcal{N}(\la)\rangle_{dx}\right|+\left|\langle\e^{-i\Th}\pa_\la\phi_S, (-(u_0\nabla_x)^2+2u_0\nabla_x X)\mathcal{N}(\la)\rangle_{dx}\right|\\
		&\les \ep^2(1+\|v\|_{L^\infty})+\ep\|\pa_\la\phi_S X^2v\|_{L^2}+\left|\langle u_0\nabla_x (\e^{-i\Th}\pa_\la\phi_S),X
		\mathcal{N}(\la)\rangle_{dx}\right| \\
		&\quad +\left|\langle(u_0\nabla_x)^2(\e^{-i\Th}\pa_\la\phi_S), \mathcal{N}(\la)\rangle_{dx}\right|\\
		&\les \ep^2(1+\|v\|_{L^\infty})+\ep\|\pa_\la\phi_S X^2v\|_{L^2} .
	\end{align*}
	Here we used the asymptotic behavior for the ground state $f_\om$ to bound $X\ln f_\om$. We hence can show that
	\begin{align*}
		|\pa_{tt}K(t;\la(t))| \les\ep^2(1+\|v\|_{L^\infty})+\ep\|\pa_\la\phi_S X^2v\|_{L^2}.
	\end{align*}
	Then Lemma \ref{nondegD}  leads to the estimate
	\[
	|\pa_{t}^3\ga|\les \ep^2(1+\|v\|_{L^\infty})+\ep\|\pa_\la\phi_S
	X^2v\|_{L^2}.
	\]
	We thus finished the proof for the Proposition.  
\end{proof}
\subsubsection{Linearized equation for $\tilde{v}$}
To derive higher order energy estimates for the remainder term $v$, we make use of the modified modulation curve  $\tilde{\la}(t)$, which is the integral curve of $V(\la)$. Under the new composition \eqref{eq:decomp:new}, we can find the equation for $\tilde{v}$  
\begin{equation}
	\label{eq:tilde:v}
	L_\ep \tilde{v}+\mathcal{N}(\tilde{\la})+\tilde{K}=0,
\end{equation}
where 
\begin{align*}
	\tilde{K}
	=-\pa_{\tilde{\la}}\phi_S(x;\tilde{\la})\cdot \pa_t V(\la)+(i\Box \Th-\pa^\mu \Th\pa_\mu\Th+\om_0^2)\tilde{v}
\end{align*}
and the linear operator $L_\ep$  
\begin{equation*}
	L_{\ep}v= \Box_{A^\ep} v-(m^2-\om_0^2) v+2i\pa_{\mu}\Th\cdot D^\mu
	v+f_{\om_0}^{p-1}(z)v+(p-1)f_{\om_0}^{p-1}(z)v_1
\end{equation*}
for any complex valued function $v=v_1+iv_2$. Since $\pa_t{\tilde{\la}(t)}=V(\la)$, we compute that 
\begin{align*}
	\pa_t z(x;\tilde{\la})&=-(u_0+u)-(\rho(0)-1)\frac{(u_0+u)\cdot (u_0+u(0))}{|u_0+u(0)|^2}(u_0+u(0)),\quad \\ 
	\pa_k z(x;\tilde{\la})&=\pa_k x+(\rho(0)-1)\frac{\pa_k x\cdot (u_0+u(0))}{|u_0+u(0)|^2}(u_0+u(0)),\\
	\pa_t\Th(\tilde{\la})&=\dot\theta-\omega_0 (u_0 +u(0))\dot z=\frac{\om}{\rho}+\om_0 \rho(0) (u+u_0)\cdot(u_0+u(0))  ,\\ 
	\pa_k\Th(\tilde{\la})&=-\om_0 (u_0+u(0)) \dot \pa_k z=-\om_0\rho(0) (u_0+u(0))_k.
\end{align*}
In particular we have 
\begin{align*}
	\Box \Th &=-\pa_t(\frac{\om}{\rho}), \\ 
	\om_0^2 -\pa^\mu \Th\pa_\mu\Th&=|\frac{\om_0}{\rho(0)}+\om_0\rho(0)|u(0)+u_0|^2 |^2 -|\frac{\om}{\rho}+\om_0\rho(0)(u+u_0)(u(0)+u_0)|^2.
\end{align*}
We therefore can derive that 
\[
|\tilde{K}|\les |\dot\ga| |\tilde{v}|\les \ep^2 |\tilde{v}|. 
\]
Now we recall the energy estimate for the covariant linear Klein-Gordon equation. Define the energy momentum tensor $\tilde{T}_{\mu\nu}[v]$ associated to the connection field $A^\ep$
\[
\tilde{T}_{\mu\nu}[v]=\l D_\mu v,D_\nu v\r-\f12 m_{\mu\nu}(\l D^\ga v, D_\ga v\r+2(m^2-\om_0^2)|v|^2).
\]
For any vector field $Y$, we have the identity
\[
\pa^\mu(\tilde{T}_{\mu\nu}[v]Y^\nu)=\tilde{T}^{\mu\nu}[v]\pi^Y_{\mu\nu}+\l \Box_{A^\ep}v-(m^2-\om_0^2)v, D_Y v\r.
\]
For the Killing vector field  $Y=X=\pa_t+u_0^k\pa_k$, integrate on the region
$[0, t]\times \mathbb{R}^3$. We obtain the energy identity
\begin{align*}
	\left. \int_{\mathbb{R}^3}\tilde{T}_{\mu 0}[v]X^\mu  d x\right|_{t=t}=\left.\int_{\mathbb{R}^3}\tilde{T}_{\mu 0}[v]X^\mu   dx \right|_{t=0}+
	\int_{0}^{t}\int_{\mathbb{R}^3} \l \Box_{A^\ep}v-(m^2-\om_0^2)v, D_X v\r dxdt.
\end{align*}
Since $X$ is timelike, we have 
\begin{align*}
	\int_{\mathbb{R}^3} |Dv|^2+|v|^2 dx\les \int_{\mathbb{R}^3}\tilde{T}_{\mu 0}[v]X^\mu  d x. 
\end{align*}
Here we may note that $m^2-\om_0^2>0$. We thus obtain the energy estimate 
\begin{align*}
	\int_{\mathbb{R}^3} |Dv|^2+|v|^2 d x \les  \int_{\mathbb{R}^3} |Dv(0, x)|^2+|v(0, x)|^2  dx +
	|\int_{0}^{t}\int_{\mathbb{R}^3} \l \Box_{A^\ep}v-(m^2-\om_0^2)v, D_X v\r dxdt|.
\end{align*}
To derive energy estimate for the linear operator $L_\ep$, we show that the other linear terms are error terms. The key observation is that the soliton travels along the direction $X$. We transfer the derivative $D_X$ to gain an extra smallness.   First using integration by parts, we can  write
\begin{align*}
	&\l 2i\pa_\mu\Th\cdot D^\mu v, D_X v\r = \pa^\mu\Th \l 2i D_\mu v,D_X v\r= \pa^\mu\Th\left(\pa_\mu\l iv, D_X v\r-X\l iv, D_\mu v\r\right)\\
	&=\pa_\mu\left( \pa^\mu\Th\l iv, D_X v\r\right)-X\left( \pa^\mu\Th \l iv, D_\mu v\r\right)-\pa_\mu \pa^\mu\Th 
	\l iv, D_X v\r+X( \pa^\mu\Th)\l iv, D_\mu v\r.
\end{align*}
The above computations indicate that 
\[
|\pa_\mu\pa^\mu\Th|+|X\pa^\mu\Th|\les  |\dot \ga|\les \ep^2. 
\]
Therefore we can bound that 
\begin{align*}
	& \left|\int_{0}^{t}\int_{\mathbb{R}^3}\l 2i\pa_\mu\Th\cdot D^\mu v, D_X v\r d xdt \right| \\
	& \les\left|\left.\int_{\mathbb{R}^3}\l iv, \pa^0\Th D_X v-\pa^\mu\Th D_\mu v\r d x\right|_{0}^{t}\right|+ \ep^2
	\int_{0}^{t}\int_{\mathbb{R}^3}|  v| |D v|d xdt \\
	&\les \|v(t, x)\|_{L^2} \|D v(t, x)\|_{L^2} +\|v(0, x)\|_{L^2} \|D v(0, x)\|_{L^2} +\ep^2 \int_{0}^{t}\|v\|_{L^2}^2+\|D  v\|_{L^2}^2ds.
\end{align*}
For the linear  terms involving the ground state, note that 
\begin{align*}
	X f_{\om_0}(z(x;\tilde{\la}(t)))= \nabla_z f_{\om_0} (\pa_t z+u_0\nabla z)=\nabla_z f_{\om_0}(u+(\rho(0)-1)\frac{u\cdot (u_0+u(0))}{|u_0+u(0)|^2}(u_0+u(0))).
\end{align*}
In particular we have 
\begin{align*}
	|Xf_{\om_0}(z(x;\tilde{\la}(t)))|\les |u|\les \ep. 
\end{align*}
We therefore can show that  
\begin{align*}
	& 2\int_{0}^{t}\int_{\mathbb{R}^3}\l  f_{\om_0}^{p-1}(v+ (p-1)  v_1),
	D_X v \r  d xdt  = \left.\int_{\mathbb{R}^3} f_{\om_0}^{p-1}( 
	|v|^2+(p-1)v_1^2)d x \right|_{0}^{t} \\ 
	&\qquad \quad -\int_{0}^{t}\int_{\mathbb{R}^3}X (f_{\om_0}^{p-1}) (|v|^2+(p-1)v_1^2) +2 (p-1)f_{\om_0}^{p-1}v_1 v_2 \l A^\ep, X \r dxdt.
\end{align*}
Here the inner product of the 1-form and the vector field is given by 
\begin{align*}
	\l A^\ep, X \r= (A^\ep)^0+(u_0)_k (A^\ep)^k= (\tilde{A}^\ep)^0+(u_0)_k (\tilde{A}^\ep)^k.
\end{align*}
Hence by using the estimate \eqref{eq:tildeA:bd} for the connection field $\tilde{A}$, we can estimate that 
\begin{align*}
	\left| 2\int_{0}^{t}\int_{\mathbb{R}^3}\l  f_{\om_0}^{p-1}(v+ (p-1)  v_1),
	D_X v \r  d xdt \right| \les  \left|\int_{\mathbb{R}^3}   
	|v|^2 d x |_{0}^t \right|  + \ep \int_{0}^{t}\int_{\mathbb{R}^3}  |v|^2  dxdt.
\end{align*}
Combining  the above estimate, we have shown that  
\begin{align*}
	&\left|\int_{0}^{t}\int_{\mathbb{R}^3}\l 2i\pa_\mu\Th\cdot D^\mu v
	+f_{\om_0}^{p-1}(z)v+(p-1)f_{\om_0}^{p-1}(z)v_1,
	D_ X v \r d x dt \right|\\
	&\les \|v(t, x)\|_{L^2} \|D v(t, x)\|_{L^2} +\|v(0, x)\|_{L^2} \|D v(0, x)\|_{H^1}+ \|v(t, x)\|_{L^2}^2   +\|v(0, x)\|_{L^2}^2\\ 
	& \quad +\ep\int_{0}^{t}\|v\|_{L^2}^2+\|Dv\|_{L^2}^2ds.
\end{align*}
By using Gronwall's inequality to absorb the last term in the above inequality (here we may note that $t\leq T/\ep$), we therefore can derive the energy estimate for the linear operator $L_\ep$
\begin{align}
	\label{eq:ee:Lep}
	\|Dv\|_{L^2}^2+\|v\|_{L^2}^2 \les    \|Dv(0, x)\|_{L^2}^2+\|v(0, x)\|_{L^2}^2   +
	\left|\int_{0}^{t}\int_{\mathbb{R}^3} \l L_\ep v, D_X v\r dxdt\right|.
\end{align}
Here we used the Cauchy-Schwarz inequality 
\begin{align*}
	2\|v(t, x)\|_{L^2} \|D v(t, x)\|_{L^2}\leq \ep_2^{-1}\|v(t, x)\|_{L^2}^2 +\ep_2 \|D v(t, x)\|_{L^2}^2
\end{align*}
to absorb the second term on the right hand side, in which the positive constant $\ep_2$ is sufficiently small. 
\subsubsection{Energy estimate for $D\tilde{v}$}
To derive the covariant energy estimate for $D\tilde{v}$, commute the equation \eqref{eq:tilde:v} with the covariant derivative $D_X$. Here recall that the vector field $X=\pa_t+u_0\nabla$.  The above energy estimate \eqref{eq:ee:Lep} indicates that we need first to compute  $L_\ep D_X\tilde{v}$. In view of the equation \eqref{eq:tilde:v} for $\tilde{v}$, we have 
\begin{align*}
	L_\ep D_X \tilde{v} & = [L_\ep, D_X] \tilde{v}+ D_X L_\ep \tilde{v} \\ 
	&=[\Box_{A^\ep}, D_X]\tilde{v}+[2i \pa_{\mu}\Th\cdot D^\mu, D_X]
	\tilde{v}- X (f_{\om_0}^{p-1}(z)) \tilde{v} \\ 
	&\quad -(p-1) X (f_{\om_0}^{p-1}(z))\tilde{v}_1-D_X( \mathcal{N}(\tilde{\la})+\tilde{K}).
\end{align*} 
First recall the commutator identity (see for example \cite{YangYu:MKG:smooth}) 
\begin{align*}
	[\Box_{A^\ep}, D_\mu]\tilde{v}&=2i F_{\nu\mu}^\ep  D^\nu \tilde{v} +i \pa^\nu  F_{\nu \mu}^\ep \tilde{v}, \\ 
	[2i \pa_{\mu}\Th\cdot D^\mu, D_X]\tilde{v}&=-2i X\pa_\mu \Th \cdot D^\mu\tilde{v}-2\pa_\mu\Th (F^\ep)^\mu_{\ X} \tilde{v}.
\end{align*}
Therefore by using the estimate \eqref{eq:tildeA:bd} for the connection field $\tilde{A}$ and the energy estimate for $\tilde{v}$ in Lemma \ref{lem:tildev}, we can bound that 
\begin{align*}
	&|\int_{0}^{t}\int_{\mathbb{R}^3} \l [\Box_{A^\ep}, D_X]\tilde{v}  , D_X D_X\tilde{v}\r dxdt| \\ 
	&\les \int_0^t \|D_XD_X\tilde{v}\|_{L^2}(\| \ep^2 D\tilde{v}\|_{L^2}+\|\pa^2\tilde{A}^\ep \|_{L^4}\|\tilde{v}\|_{L^4}) ds \\ 
	&\les \int_0^t \|D_XD_X\tilde{v}\|_{L^2}( \ep^3 +\|\pa^2\tilde{A}^\ep \|_{H^1}\|\tilde{v}\|_{H^1}) ds \\ &\les  \ep^3 \int_0^t \|D_XD_X\tilde{v}\|_{L^2} ds. 
\end{align*}
Next note that 
\begin{align*}
	|X\pa \Th|\les |\dot\gamma|\les \ep^2,\quad |X(f_{\om_0})|\les \ep. 
\end{align*}
Using the bound for $\tilde{v}$ and $\tilde{v}$ in Lemma \ref{lem:tildev}, we then can show that 
\begin{align*}
	&|\int_{0}^{t}\int_{\mathbb{R}^3} \l [2i X\pa_{\mu}\Th\cdot D^\mu, D_X]
	\tilde{v}- X (f_{\om_0}^{p-1}(z)) \tilde{v}-(p-1) X (f_{\om_0}^{p-1}(z))\tilde{v}_1 , D_X D_X\tilde{v}\r dxdt| \\ 
	&\les \int_0^t \|D_XD_X\tilde{v}\|_{L^2}(\| \ep^2 D\tilde{v}\|_{L^2}+ \ep \|\tilde{v}\|_{L^2}) ds \\ 
	&\les   \ep^2 \int_0^t \|D_XD_X\tilde{v}\|_{L^2} ds. 
\end{align*}
By definition of $\tilde{K}$ and the bound for $\dot \gamma$ and $\ddot{\gamma}$, we estimate that 
\begin{align*}
	|D_X\tilde{K} | 
	\les \ep^2(|D_X\pa_{\tilde{\la}}\phi_S |+ |\pa_{\tilde{\la}}\phi_S|)+ \ep^2(|D_X\tilde{v}|+|\tilde{v}|).
\end{align*}
This shows that 
\begin{align*}
	|\int_{0}^{t}\int_{\mathbb{R}^3} \l  D_X\tilde{K}, D_X D_X\tilde{v}\r dxdt|  \les \ep^2\int_0^t \|D_XD_X\tilde{v}\|_{L^2} ds. 
\end{align*}
Finally for the nonlinearity $\mathcal{N}$, in view of Lemma \ref{lem:Y:N}, we can estimate that 
\begin{align*}
	|\int_{0}^{t}\int_{\mathbb{R}^3} \l  D_X \mathcal{N}, D_X D_X\tilde{v}\r dxdt|
	&\les \int_0^t \|D_X D_X\tilde{v}\|_{L^2}  \| (|\tilde{v}|+|\tilde{v}|^{p-1})(|D_{X}
	f_\om|+|D_X \tilde{v}|+|\tilde{v}|)\|_{L^2} ds \\ 
	&\les \int_0^t \|D_X D_X\tilde{v}\|_{L^2} ( \ep^2+\ep  \|D_X \tilde{v}\|_{L^4})   ds\\ 
	&\les  \int_0^t \|D_X D_X\tilde{v}\|_{L^2} ( \ep^2+\ep  \|D D_X \tilde{v}\|_{L^2})   ds.
\end{align*}
Combining these estimates, we have shown that 
\begin{align*}
	|\int_{0}^{t}\int_{\mathbb{R}^3} \l  L_\ep D_X \tilde{v} , D_X D_X\tilde{v}\r dxdt|\les  \int_0^t \|D_X D_X\tilde{v}\|_{L^2} ( \ep^2+\ep  \|D D_X \tilde{v}\|_{L^2})   ds.
\end{align*}
Then the energy estimate \eqref{eq:ee:Lep} leads to 
\begin{align*}
	\|DD_X\tilde{v}\|_{L^2}^2 & \les \ep^2+\|D_X\tilde{v}\|_{L^2}^2+|\int_0^t \int_{\mathbb{R}^3}  \l  L_\ep D_X \tilde{v} , D_X D_X\tilde{v}\r dxdt|\\ 
	&\les \ep^2+ \int_0^t   \|D_X D_X\tilde{v}\|_{L^2} ( \ep^2+\ep  \|D D_X \tilde{v}\|_{L^2})    ds.
\end{align*}
Gronwall's inequality then implies that 
\begin{align*}
	\|DD_X\tilde{v}\|_{L^2}^2 \les \ep^2,\quad \forall t\leq T/\ep. 
\end{align*}
To derive the energy estimate for other derivatives of the solution, we rely on elliptic estimate. 
First note that 
\begin{align*}
	D_0 D_0 
	=D_X D_X-u_0^k (2 D_k D_X- iF^\ep(X, \pa_k))+|u_0|^2  D_3 D_3. 
\end{align*}
Here we used the fact that $u_0$ is parallel to $(0, 0, 1)$. In particular we can   rewrite the equation \eqref{eq:tilde:v} for $\tilde{v}$ as 
\begin{align*}
	D_1D_1 \tilde{v}+D_2 D_2\tilde{v}+ (1-|u_0|^2)D_3 D_3\tilde{v} &= (m^2-\om_0^2) \tilde{v}-2i\pa_{\mu}\Th\cdot D^\mu
	\tilde{v}-f_{\om_0}^{p-1}(z)\tilde{v}-(p-1)f_{\om_0}^{p-1}(z)\tilde{v}_1\\ 
	&\quad -\mathcal{N}(\tilde{\la})-\tilde{K} +D_X D_X\tilde{v}-u_0^k (2 D_k D_X\tilde{v}- iF^\ep(X, \pa_k)\tilde{v}).
\end{align*}
Since $|u_0|<1$, elliptic estimate then implies that 
\begin{align*}
	\|D_jD_k\tilde{v}\|_{L^2} & \les \| D_1D_1 \tilde{v}+D_2 D_2\tilde{v}+ (1-|u_0|^2)D_3 D_3\tilde{v} \|_{L^2}+\| F^\ep D\tilde{v}\|_{L^2} \\ 
	&\les \ep+ \|  \tilde{v}\|_{L^2} +\| D
	\tilde{v}\|_{L^2} +\|  \mathcal{N}(\tilde{\la})\|_{L^2}+\|\tilde{K}\|_{L^2} +\| D D_X\tilde{v}\|_{L^2}  \\ 
	&\les \ep+  \|  |\tilde{v}|^2+|\tilde{v}|^{p+1}\|_{L^2}\\ 
	&\les \ep.   
\end{align*} 
We remark here that the above elliptic estimate can obtained by continuity argument or  through  integration by parts directly. The commutator of the covariant derivatives $D$ is the Maxwell field $F^\ep$, which is uniformly small in view of the condition \eqref{eq:tildeA:bd} for the connection field $A$. We therefore have obtained the energy estimate for $D\tilde{v}$
\begin{equation}
	\label{eq:EE:v:2}
	\|DD\tilde{v}\|_{L^2}\les \ep, \quad \forall t\leq T/\ep.
\end{equation} 
\subsubsection{Energy estimate for $DD\tilde{v}$}
We proceed in a similar way to show the third order energy estimate for $\tilde{v}$, for which we need first bound the third order derivatives of the modulation curve $\gamma(t)$. In view of  Proposition \ref{prop:higher:mod} and the above second order energy estimate for $\tilde{v}$, we now are able to derive the $L^\infty$ bound for $v$.  From the decomposition \eqref{eq:decomp} corresponding to the original modulation curve $\la(t)$, by using Sobolev embedding we can show that
\begin{align*}
	\|v\|_{L^\infty}&\les  \sum\limits_{k\leq 2}\|D^k v\|_{L^2}\les \sum\limits_{k\leq 2}\| D^k \left(  \e^{-i\Th(\la)}(\phi_S(\la;x)-\phi_S(\tilde{\la};x)-\e^{i\Th(\tilde{\la})}\tilde{v})\right)\|_{L^2}\\
	&\les \sum\limits_{k\leq 2}|(\pa_t)^k\la(t)-(\pa_t)^k\tilde{\la}(t)|+\sum\limits_{k\leq 2}\|D^k\tilde{v}\|_{L^2} \\
	&\les   \ep + |\dot \gamma|+|\ddot{\gamma}|\les \ep.  
\end{align*}
Since the soliton $\phi_S$ decays exponentially, the above estimate also indicates that  
\begin{align*}
	\|\pa_\la\phi_S X^2 v\|_{L^2}& \les \|D_X D_X v\|_{L^2}+\|  \pa_\la \phi_S \left(i \l A^\ep, X\r( 2 D_X v -i\l A^\ep, X\r v)+i X \l A^\ep, X\r  v \right)\|_{L^2}\\
	&\les   \|D D v\|_{L^2}+\|v\|_{L^2} +\|  |\pa_\la \phi_S|  \left( |\tilde{A}^\ep|  | D_X v|+|\tilde{A}^\ep|^2 |v|  \right)\|_{L^2} \\ 
	&\les   \|D D v\|_{L^2}+\|D v\|_{L^2}+ \|v\|_{L^2} \les \ep.  
\end{align*}
Here we used the fact that 
\[
|\l A^\ep, X\r |=|(\tilde{A}^\ep)^0+(u_0)_k (\tilde{A}^\ep)^k |\les \ep. 
\] 
Therefore  Proposition \ref{prop:higher:mod} leads to the improved bound  
\begin{equation*}
	|\pa_t^3 \ga|\les
	\ep^2(1+\|v\|_{L^\infty})+\ep\|\pa_\la\phi_S X^2v\|_{L^2}\les
	\ep^2.
\end{equation*}
To show energy estimate for $DD\tilde{v}$, we commute the equation \eqref{eq:tilde:v} for $\tilde{v}$ with $D_X$ twice. We first compute that 
\begin{align*}
	L_\ep D_XD_X \tilde{v} & = [L_\ep, D_XD_X] \tilde{v}+ D_X D_X L_\ep \tilde{v} \\ 
	&=[\Box_{A^\ep}, D_XD_X]\tilde{v}+[2i \pa_{\mu}\Th\cdot D^\mu, D_X D_X]
	\tilde{v}+ [f_{\om_0}^{p-1}, D_X D_X] \tilde{v} \\ 
	&\quad +(p-1)  [f_{\om_0}^{p-1}, D_XD_X]\tilde{v}_1-D_XD_X( \mathcal{N}(\tilde{\la})+\tilde{K}).
\end{align*} 
We have the covariant  commutator identity 
\begin{align*}
	[\Box_{A^\ep}, D_\mu D_\nu]&=2i(F_{l\mu}^\ep D^l D_{\nu}+F_{l\nu}^\ep D^l D_{\mu})+i(\pa^l F_{l\mu}^\ep D_\nu+\pa^l F_{l\nu}^\ep D_\mu)\\ 
	&\quad +2i \pa_\mu F_{l\nu}^\ep D^l+ i \pa_\mu \pa^l F_{l\nu}^\ep.
\end{align*}
Thus  by using the estimate \eqref{eq:tildeA:bd} for the connection field $\tilde{A}$ and the energy estimate for $\tilde{v}$, $D\tilde{v}$ obtained previously, we can bound that 
\begin{align*}
	\|  [\Box_{A^\ep}, D_XD_X]\tilde{v}\|_{L^2}  & \les \ep^2 \|  D  D_{X}\tilde{v}\|_{L^2}+  \| \pa^2 \tilde{A}^\ep  D \tilde{v}\|_{L^2}+ \|\pa^3\tilde{A}^\ep \tilde{v} \|_{L^2} \\ 
	&\les \ep^2 \|  D  D \tilde{v}\|_{L^2}+  \| \pa^2 \tilde{A}^\ep \|_{L^4} \| D \tilde{v}\|_{L^4}+ \|\pa^3\tilde{A}^\ep \|_{L^2} \|\tilde{v} \|_{L^\infty} \\ 
	&\les   \ep^3 +  \| \pa^2 \tilde{A}^\ep \|_{L^4} \| D \tilde{v}\|_{L^4}+ \|\pa^3\tilde{A}^\ep \|_{L^2} \|\tilde{v} \|_{L^\infty}\\ 
	& \les \ep^3.
\end{align*}
Similarly we can compute that 
\begin{align*}
	[\pa_{\mu}\Th\cdot D^\mu, D_X D_X]\tilde{v}&=\pa_\mu\Th [D^\mu, D_X D_X]\tilde{v}-X^2\pa_\mu\Th \cdot D^\mu \tilde{v}-2X\pa_\mu\Th \cdot D_X D^\mu\tilde{v}\\ 
	& =i\pa_\mu\Th (2(F^\ep)^\mu_{\ X}D_X\tilde{v}+X (F^\ep)^\mu_{\ X}\tilde{v}) -X^2\pa_\mu\Th \cdot D^\mu \tilde{v}-2X\pa_\mu\Th \cdot D_X D^\mu\tilde{v}.
\end{align*}
Thus we can estimate that 
\begin{align*}
	\|  [2i\pa_{\mu}\Th\cdot D^\mu, D_X D_X]\tilde{v} \|_{L^2}  & \les \ep^2 \|    D_{X}\tilde{v}\|_{L^2}+  \| \pa^2 \tilde{A}^\ep  \tilde{v}\|_{L^2}+ \| (|\dot\ga|^2+|\ddot{\gamma}|)  D^\mu \tilde{v} \|_{L^2} +  \|  |\dot \gamma |\tilde{v}\|_{L^2}\\ 
	&\les   \ep^3 +  \| \pa^2 \tilde{A}^\ep \|_{L^4} \| D \tilde{v}\|_{L^4} \\ 
	& \les \ep^3.
\end{align*}
Next notice that 
\begin{align*}
	|X( z(x;\tilde{\la}))|\les \ep,\quad [f_{\om_0}^{p-1}, D_XD_X] =-X^2 (f_{\om_0}^{p-1}) \tilde{v}-2X(f_{\om_0}^{p-1})D_X\tilde{v}.
\end{align*}
Hence we have 
\begin{align*}
	\|  [ f_{\om_0}^{p-1}, D_X D_X]\tilde{v} \|_{L^2} +\|[f_{\om_0}^{p-1}, D_XD_X]\tilde{v}_1\|_{L^2}  & \les \ep (  \|    D_{X}\tilde{v}\|_{L^2}+      \|   \tilde{v}\|_{L^2})  \les \ep^2.
\end{align*}
For  $D_XD_X\tilde{K}$,  we rely on the above estimates  for $\dot \gamma$, $\ddot{\gamma}$ and $\pa_t^3\gamma$. We show that 
\begin{align*}
	|D_XD_X\tilde{K}|&=\left|D_X D_X\left(-\pa_{\tilde{\la}}\phi_S(x;\tilde{\la})\cdot \pa_t V(\la)+ (i\Box \Th-\pa^\mu \Th\pa_\mu\Th+\om_0^2)\tilde{v} \right)\right|\\ 
	&\les  \ep^2 (|\tilde{v}|+|D_X\tilde{v}|+|D_XD_X\tilde{v}|+|D_X D_X \pa_{\tilde{\la}}\phi_S|+|  D_X \pa_{\tilde{\la}}\phi_S|+| \pa_{\tilde{\la}}\phi_S|),
\end{align*}
which in particular implies that 
\begin{align*}
	\|D_XD_X\tilde{K}\|_{L^2}\les \ep^2. 
\end{align*}
Finally for the nonlinearity $\mathcal{N}$, in view of the asymptotic behavior of the ground state $f_{\om_0}$, we first show that 
\begin{align*}
	|D_X f_{\om_0}|&\les |X f_{\om_0}|+|A^\ep|f_{\om_0}\les \ep +\ep (1+|x|) f_{\om_0}\les \ep + \ep (1+|z|+|\eta|) f_{\om_0}\les \ep,\\ 
	|D_X D_X f_{\om_0}| & \les |X^2 f_{\om_0}| +|A^\ep| |Xf_{\om_0}|+|A^\ep|^2 f_{\om_0} \les \ep. 
\end{align*}
Thus by 
using Lemma \ref{lem:Y:N},  we can estimate that 
\begin{align*}
	\|    D_XD_X\mathcal{N}\|_{L^2} 
	&\les  \| (1+\|\tilde{v}\|_{L^\infty}+ |D_{X} \ln
	f_{\om_0}| )( |D_{X}  f_{\om_0}|^2 +|D_{X}
	\tilde{v}|^2 +|\tilde{v}|^2) \|_{L^2} \\ 
	& \quad +\|(|\tilde{v}|+|\tilde{v}|^{p-1})(|D_{X}D_{X} f_{\om_0}|+|D_{X}D_{X}\tilde{v}|) \|_{L^2} \\ 
	&\les  \ep^2+\|   D_{X}
	\tilde{v}\|_{L^4}^2 +\|\tilde{v} \|_{L^4}^2+ \ep \| |D_{X}D_{X} f_{\om_0}| +|D_{X}D_{X}\tilde{v}|  \|_{L^2} \\ 
	&\les \ep^2.
\end{align*}
Combining all the above estimates, we have shown that 
\begin{align*}
	\|   L_\ep D_X D_X\tilde{v} \|_{L^2}\les  \ep^2. 
\end{align*}
Then the energy estimate \eqref{eq:ee:Lep} indicates that 
\begin{align*}
	\|DD_XD_X\tilde{v}\|_{L^2}^2 & \les \ep^2+|\int_0^t \int_{\mathbb{R}^3}  \l  L_\ep D_X D_X\tilde{v} , D_X D_XD_X\tilde{v}\r dxdt|\\ 
	&\les \ep^2+ \ep^2\int_0^t   \|D_X D_XD_X\tilde{v}\|_{L^2}     ds,
\end{align*}
which, combined with Gronwall's inequality leads to the second order energy estimate  
\begin{align*}
	\|DD_XD_X\tilde{v}\|_{L^2}^2 \les \ep^2,\quad \forall t\leq T/\ep. 
\end{align*}
Now we go back to the equation for $D_X\tilde{v}$ and make use of similar elliptic estimate. We can show that 
\begin{align*}
	\|D_jD_kD_X\tilde{v}\|_{L^2} & \les \| L_\ep D_X\tilde{v}\|_{L^2}+ \|(L_\ep-\Box_{A^\ep})D_X \tilde{v}\|_{L^2}+\|  DD_XD_X\tilde{v}\|_{L^2}+\||F^\ep|D_X\tilde{v}\|_{L^2} \\ 
	&\les \ep+  (\ep^2+\ep\|DD_X\tilde{v}\|_{L^2}) + \|  D_X \tilde{v}\|_{L^2} +\|  D^\mu D_X
	\tilde{v}\|_{L^2}\\ 
	&\les \ep.  
\end{align*}
Here during the proof for the energy estimate for $D_X\tilde{v}$ in the previous section, we have shown that 
\[
\| L_\ep D_X\tilde{v}\|_{L^2}\les \ep^2+\ep\|DD_X\tilde{v}\|_{L^2}\les \ep^2. 
\]
We therefore can conclude that  
\begin{equation*}
	\|DDD_X\tilde{v}\|_{L^2}\les \ep, \quad \forall t\leq T/\ep.
\end{equation*} 
To derive the full third order energy estimate for $\tilde{v}$, that is $D^3\tilde{v}$, commute the equation \eqref{eq:tilde:v} for $\tilde{v}$ with the covariant derivative $D$. We can estimate that 
\begin{align*}
	\| \Box_{A^\ep} D \tilde{v}\|_{L^2} &\les   \| [\Box_{A^\ep}, D ]\tilde{v}\|_{L^2}+ \|D ( \mathcal{N}(\tilde{\la})+\tilde{K})\|_{L^2}+ (m^2-\om_0^2)\|D\tilde{v}\|_{L^2} \\ 
	&\quad +\|D(2i\pa_{\mu}\Th\cdot D^\mu
	v+f_{\om_0}^{p-1}(z)v+(p-1)f_{\om_0}^{p-1}(z)v_1)\|_{L^2} \\ 
	&\les \ep+ \| 2i F_{\nu\mu}^\ep  D^\nu \tilde{v} +i \pa^\nu  F_{\nu \mu}^\ep \tilde{v} \|_{L^2}+ \|(|\tilde{v}|+|\tilde{v}|^{p-1})(|D 
	f_\om|+|D \tilde{ v}|+|\tilde{v}|) \|_{L^2} \\ 
	&\quad + \| D(-\pa_{\tilde{\la}}\phi_S(x;\tilde{\la})\cdot \pa_t V(\la) )\|_{L^2} \\ 
	&\les \ep + \|   D^\nu \tilde{v} \|_{L^2} +  \|\pa^2 \tilde{A}^\ep  \|_{L^2}+ \ep \| |D 
	f_\om|+|D \tilde{ v}|+|\tilde{v}| \|_{L^2} + |\dot\gamma| +|\ddot{\gamma}| \les \ep .
\end{align*} 
Now by using elliptic estimate again, we have 
\begin{align*}
	\|D_j D_k D\tilde{v}\|_{L^2}&\les \|\Box_{A^\ep}D\tilde{v}\|_{L^2}+\|D_X D_X D\tilde{v}\|_{L^2}+\|u_0^k (2 D_k D_X- iF^\ep(X, \pa_k))D\tilde{v}\|_{L^2} \\ 
	&\les \ep + \|D_X D D_X \tilde{v}\|_{L^2}+\|D [D,  D_X]\tilde{v}\|_{L^2}+\|DD D_X\tilde{v}\|_{L^2} \\ 
	&\les \ep +\|D(i F^\ep \tilde{v})\|_{L^2}\les \ep. 
\end{align*}
This together with the above estimate for $DDD_X\tilde{v}$ is sufficient to conclude the third order energy estimate 
\[
\|DDD\tilde{v}\|_{L^2}\les \ep, \quad t\leq T/\ep. 
\]
We thus have shown Proposition \ref{pro:phi:HSob}. 
\section{Energy estimates for electromagnetic field}\label{recover EM}
Under the Lorentz gauge condition, the connection field $\tilde{A}^{\ep}=A^\ep-A^\ep_b$ verifies the following wave equation 
\begin{align*}
	\pa^\mu F^\ep_{\mu\nu}=\Box A^{\ep}_{\nu}-\pa_\nu\pa^\mu A^\ep_\mu =\Box \tilde{A}^\ep =-\delta^2\ep^2 J[\phi]_{\nu}.
\end{align*}
Since the initial data are given in terms of the initial electric field and the magnetic field together with the gauge invariant norm of the scalar field, to obtain energy estimates for the connection field $\tilde{A}^\ep$, we need first assign appropriate initial data for $\tilde{A}^\ep$, which is consistent with the Lorentz gauge condition. For this purpose, note that for any free wave solution $\chi$ such that $\Box\chi=0$, the new connection field $\tilde{A}-d\chi$ still verifies the Lorentz gauge condition. For any given connection field $\tilde{A}^\ep$ verifying the Lorentz gauge condition, let $\chi$ be the solution to the Cauchy problem of the linear wave equation 
\begin{align*}
	\Box\chi=0, \quad \chi(0, x)=\Delta^{-1} \pa_t\tilde{A}^\ep_0(0, x),\quad \pa_t\chi(0, x)=\Delta^{-1}\pa_t \pa_j \tilde{A}^\ep_j(0, x). 
\end{align*}
For this new connection field $\tilde{A}^\ep - d\chi$, it holds that 
\begin{align*}
	\pa_t(\tilde{A}^\ep_0-\pa_t\chi)|_{t=0}=0,\quad \pa_t \pa_j(\tilde{A}^\ep_j-\pa_j\chi)|_{t=0}=0.
\end{align*}
In particular for the initial data for the connection field $\tilde{A}^\ep$, we may require that on the initial hypersurface $\{t=0\}$ it holds that 
\begin{align*}
	\pa_t (\tilde{A}^\ep)_0(0, x)=0,\quad \pa_t\pa^j (\tilde{A}^\ep)_j(0, x)=0. 
\end{align*}
Thus the Lorentz gauge condition also indicates that 
\begin{align*}
	\pa_j (\tilde{A}^\ep)^j(0, x)=\pa_t(\tilde{A}^\ep)_0(0, x)=0. 
\end{align*}
Now the compatibility condition for the electric field $E^\ep(0)$ shows that  
\begin{align*}
	div(E^\ep(0)) &= \pa_j(\pa_t A^\ep_j-\pa_j A^\ep_0)=\pa_j(\pa_t (\tilde{A}^\ep)_j-\pa_j \tilde{A}^\ep_0)=-\Delta (\tilde{A}^\ep)_0=-\delta^2\ep^2 \Im(\phi_0\cdot\overline{\phi}_1).
\end{align*}
Since $\phi_0, \phi_1$ is bounded in the weighted energy estimate, standard elliptic estimate indicates that 
\begin{align*}
	\|\tilde{A}^\ep_0(0, x)\|_{L^\infty}+\|\nabla^{k+1} \tilde{A}^\ep_0(0, x)\|_{L^2}\leq 10\delta^2\ep^2\leq 10\ep^2,\quad  k\leq 2. 
\end{align*}
By the definition of magnetic field, we have 
\begin{align*}
	\nabla\times (\nabla\times \bar{\tilde{A}}^\ep)= \nabla\times (B^\ep(0)-\ep^2(0, 0, 1))=-\Delta \bar{\tilde{A}}^\ep+\nabla (\nabla \cdot \bar{\tilde{A}}^\ep)= -\Delta \bar{\tilde{A}}^\ep.
\end{align*}
Here $\bar{\tilde{A}}^\ep= (\tilde{A}^\ep_1, \tilde{A}^\ep_2, \tilde{A}^\ep_3)$. Again elliptic estimate together with the assumption on the initial magnetic field implies that 
\begin{align*}
	\| \tilde{A}^\ep_j(0, x)\|_{L^\infty}+\|\nabla^{k+1} \tilde{A}^\ep_j(0, x)\|_{L^2}\les \delta^2\ep^2,\quad  k\leq 2. 
\end{align*}
For the time derivative of the connection field $\tilde{A}^\ep_j$ on the initial hypersurface, using the definition of the electric field, we have 
\begin{align*}
	\|\nabla^k \pa_t \tilde{A}^\ep_j(0, x)\|_{L^2}\leq \|\nabla^k(\pa_t \tilde{A}^\ep_j-\pa_j \tilde{A}^\ep_0)\|_{L^2}+\|\nabla^k \pa_j \tilde{A}^\ep_0 \|_{L^2}\leq \|\nabla^k E^\ep(0)\|_{L^2}+\|\nabla^{k+1} \tilde{A}^\ep_0 \|_{L^2}\les \delta^2\ep^2. 
\end{align*}
This in particular means that the bootstrap assumption \eqref{eq:cond4:A} for the connection field $\tilde{A}^\ep$ holds initially at $t=0$. Next we rely on the above wave equation for the connection field $\tilde{A}^\ep$ together with the higher order Sobolev norms for the scalar field obtained in Proposition \ref{pro:phi:HSob} to improve the bootstrap assumption \eqref{eq:cond4:A}. 

Note that the vector field $X=\pa_t+ (u_0)^k\pa_k$ is timelike. Using the vector field $X$ as multiplier, we obtain the energy estimate for the connection field $\tilde{A}^\ep$
\begin{align*}
	\|\pa \pa^l \tilde{A}^\ep \|_{L_x^2}^2 \les \|\pa\pa^l \tilde{A}^\ep(0, x)\|_{L_x^2}^2 +\delta^2\ep^2 |\int_0^t \int_{\mathbb{R}^3} \pa^l J[\phi]_{\nu} \cdot X \pa^l (\tilde{A}^\ep)_{\nu} dxdt|.
\end{align*}
First we have the identity 
\begin{align*}
	\pa^l J[\phi]_{\nu}=\sum\limits_{l_1+l_2=l} \Im(D^{l_1}\phi\cdot \overline{D^{l_2}D_\nu\phi}).
\end{align*}
According to the decomposition \eqref{eq:decomp:new} associated to the new modulation curve $\tilde{\la}(t)$, we have 
\begin{align*}
	&|\Im(D^{l_1}\phi\cdot \overline{D^{l_2}D_\nu\phi})- \Im(D^{l_1}\phi_S(x;\tilde{\la}(t))\cdot \overline{D^{l_2}D_\nu\phi_S(x;\tilde{\la}(t))})|\\ 
	&\les |D^{l_1}(e^{i\Th}\tilde{v})| |D^{l_2}D_\nu \phi_S|+ |D^{l_1} \phi_S| |D^{l_2}D_\nu (e^{i\Th}\tilde{v})|+|D^{l_1}(e^{i\Th}\tilde{v})| |D^{l_2}D_\nu (e^{i\Th}\tilde{v}) |.
\end{align*}
Since 
\[
\pa_t{\tilde{\la}}=V(\la),\quad |\dot\gamma|+|\ddot{\gamma}|+|\pa_t^3\gamma|\les \ep^2,\quad \forall t\leq T/\ep,
\]
it holds that 
\begin{align*}
	|\pa^k\Th|\les 1,\quad k\leq 4.
\end{align*}
In view of the higher order energy estimates for $\tilde{v}$ in Proposition \ref{pro:phi:HSob} and the bootstrap assumption \eqref{eq:cond4:A} for the connection field $\tilde{A}^\ep$, we have 
\begin{align*}
	\|D\tilde{v}\|_{L^\infty}+\|\tilde{v}\|_{L^\infty} &\les \ep,\quad |Df_{\om_0}|\les |\pa f_{\om_0}|+ |A^\ep |f_{\om_0}\les 1, \\ 
	|D^2 f_{\om_0}|&\les |\pa^2 f_{\om_0}|+|A^\ep| |\pa f_{\om_0}|+(|\pa A^\ep|+|A^\ep|^2)f_{\om_0}\les 1. 
\end{align*}
Thus for $l_1+l_2\leq 2$, we can bound that 
\begin{align*}
	&\sum\limits_{l_1+l_2\leq 2}\|\Im(D^{l_1}\phi\cdot \overline{D^{l_2}D_\nu\phi})- \Im(D^{l_1}\phi_S(x;\tilde{\la}(t))\cdot \overline{D^{l_2}D_\nu\phi_S(x;\tilde{\la}(t))})\|_{L^2}\\ 
	&\les \sum\limits_{l_1+l_2\leq 2} \| |D^{l_1} \tilde{v} | |D^{l_2+1} f_{\om_0}| \|_{L^2}+\| |D^{l_1} f_{\om_0}| |D^{l_2+1} \tilde{v}\|_{L^2}+\| |D^{l_1} \tilde{v} | |D^{l_2+1} \tilde{v} | \|_{L^2} \\ 
	&\les \ep^2+\ep \sum\limits_{l\leq 3}\|D^l f_{\om_0}\|_{L^2}  \les \ep +\ep\|\pa^2\tilde{A}^\ep\|_{L^2}\les \ep. 
\end{align*}
Here we may note that although the connection field $A^\ep_b$ grows linearly, the ground state $f_{\om_0}$ decays exponentially and the center of the soliton verifies the bound 
$$|x_1|+|x_2|\les |\eta|+|z|\les 1+|z|. $$ 
It is the soliton part which we need further consideration. For $l_1+l_2=l\leq 2$, integration by parts leads to  
\begin{align*}
	&\left|\int_0^t \int_{\mathbb{R}^3}  \Im(D^{l_1}\phi_S\cdot \overline{D^{l_2}D_\nu\phi_S}) \cdot X \pa^l (\tilde{A}^\ep)_{\nu} dxdt\right| \\ 
	&\les \left|  \int_{\mathbb{R}^3}  \Im(D^{l_1}\phi_S\cdot \overline{D^{l_2}D_\nu\phi_S}) \cdot   \pa^l (\tilde{A}^\ep)_{\nu} dx |_{t=0}^{t}\right| +\left|\int_0^t \int_{\mathbb{R}^3}  X\Im(D^{l_1}\phi_S\cdot \overline{D^{l_2}D_\nu\phi_S}) \cdot \pa^l (\tilde{A}^\ep)_{\nu} dxdt\right|. 
\end{align*}
For the boundary terms, using Hardy's inequality and the bootstrap assumption \eqref{eq:cond4:A}, we show that 
\begin{align*}
	& \left|\int_{\mathbb{R}^3}  \Im(D^{l_1}\phi_S\cdot \overline{D^{l_2}D_\nu\phi_S}) \cdot   \pa^l (\tilde{A}^\ep)_{\nu} dx\right| \\ 
	&\les \|(1+|x-\xi(t)|)^{-1}\pa^l \tilde{A}^\ep \|_{L^2} \|(1+|z|) |D^{l_1} f_{\om_0}| |D^{l_2+1}f_{\om_0}|\|_{L^2}\\ 
	&\les \| \pa^{l+1} \tilde{A}^\ep \|_{L^2}\les \ep^2.
\end{align*}
Since 
\begin{align*}
	X\Im(D^{l_1}\phi_S\cdot \overline{D^{l_2}D_\nu\phi_S})=\Im(D_XD^{l_1}\phi_S\cdot \overline{D^{l_2}D_\nu\phi_S})+\Im(D^{l_1}\phi_S\cdot \overline{D_XD^{l_2}D_\nu\phi_S})
\end{align*}
and recalling  that 
$$|X(z(x;\tilde{\la}(t)))|=\left|(u+(\rho(0)-1)\frac{u\cdot (u_0+u(0))}{|u_0+u(0)|^2}(u_0+u(0)))\right|\les \ep,$$
for $l_1+l_2\leq 3$, we can estimate that  
\begin{align*}
	&|\Im(D_XD^{l_1}\phi_S\cdot \overline{D^{l_2} \phi_S})| \\ 
	& \les |D_XD^{l_1}f_{\om_0}| |D^{l_2} f_{\om_0}| \\ 
	&\les |D^{l_1}D_X f_{\om_0}| |D^{l_2} f_{\om_0}|  +|[D_X, D^{l_1}]f_{\om_0}||D^{l_2} f_{\om_0}| \\ 
	&\les \left(  |D^{l_1} ( \nabla_z f_{\om_0} X(z)+i\l A^\ep, X\r f_{\om_0})| +\sum\limits_{l\leq l_1-1} |\pa^l F^\ep| |D^{l_1-l}f_{\om_0}|\right) |D^{l_2} f_{\om_0}| \\ 
	&\les \left(  \sum\limits_{l\leq l_1}\ep |D^{l} \nabla_z f_{\om_0} | + |\pa^l\tilde{A}^\ep| |D^{l_1-l}f_{\om_0}| +\sum\limits_{l \leq l_1-1 } |\pa^{l+1}  \tilde{A}^\ep| |D^{l_1-l}f_{\om_0}| \right) |D^{l_2} f_{\om_0}| \\ 
	&\les \left(  \sum\limits_{l\leq l_1}\ep |D^{l} \nabla_z f_{\om_0} |  +\sum\limits_{l \leq l_1-1 } |\pa^{l+1}  \tilde{A}^\ep| |D^{l_1-l}f_{\om_0}| \right) |D^{l_2} f_{\om_0}|. 
\end{align*}
We emphasize here that this is the reason we need to use the new modulation curve $\tilde{\la}$ in order to avoid the fourth order derivative of the original modulation curve $\la(t)$. 
Now note that $$|z|\les |x-\xi(t)|\les |z|.$$
 We also have the bound 
\begin{align*}
	\| (1+|x-\xi(t)|) D^l f_{\om_0}(z)\|_{L^2}\les 1,\quad \forall l\leq 4.
\end{align*}
We therefore can bound that 
\begin{align*}
	&\|(1+|x-\xi(t)|) |\Im(D_XD^{l_1}\phi_S\cdot \overline{D^{l_2} \phi_S})|\|_{L^2} \\ 
	&\les \ep \sum\limits_{l\leq 3}\|(1+|x-\xi(t)|)D^l \nabla_z f_{\om_0}\|_{L^2}+ \|(1+|x-\xi(t)|)D^l f_{\om_0}\|_{L^2}\\ 
	&+\sum\limits_{0\leq l\leq l_1-1}\|\pa^{l+1}\tilde{A}^\ep (1+|x-\xi(t)|)D^{l_1-l}f_{\om_0}D^{l_2 }f_{\om_0} \|_{L^2}\\ 
	&\les \ep +\sum\limits_{l\leq 2}\|\pa^{l+1}\tilde{A}^\ep   \|_{L^2}\les \ep. 
\end{align*}
Here we used the fact that 
\begin{align*}
	(1+|x-\xi(t)|) |D^l f_{\om_0}(z)|\les 1,\quad \forall l\leq 2 
\end{align*}
and the last step follows from the bootstrap assumption \eqref{eq:cond4:A}. Therefore we conclude that 
\begin{align*}
	& |\int_{\mathbb{R}^3}  X\Im(D^{l_1}\phi_S\cdot \overline{D^{l_2}D_\nu\phi_S}) \cdot \pa^l (\tilde{A}^\ep)_{\nu} dx| \\ 
	&\les \| (1+|x-\xi(t)|^{-1})\pa^l \tilde{A}^\ep\|_{L^2} \|(1+|x-\xi(t)|^{-1}) X\Im(D^{l_1}\phi_S\cdot \overline{D^{l_2}D_\nu\phi_S}) \| \\ 
	&\les  \sum\limits_{l_1+l_2\leq l}\| \pa^{l+1} \tilde{A}^\ep\|_{L^2} \|(1+|x-\xi(t)|) |\Im(D_XD^{l_1}\phi_S\cdot \overline{D^{l_2} \phi_S})|\|_{L^2} \\ 
	&\les \ep \sum\limits_{l \leq 2}\| \pa^{l+1} \tilde{A}^\ep\|_{L^2}\les \ep^3.
\end{align*}
Combining all the above estimates, the energy estimate for $\pa^l \tilde{A}^\ep$ then leads to 
\begin{align*}
	\|\pa \pa^l \tilde{A}^\ep \|_{L_x^2}^2 &\les \delta^4\ep^4  +\delta^2\ep^2 (\ep^3 +\int_0^t \ep \|X\pa^l \tilde{A}^\ep\|_{L^2} +\ep^3 ds ) \\ 
	&\les \delta^2 \ep^2(\ep^3+(T/\ep) \ep^3 )\les \delta^2 \ep^4 .
\end{align*}
For sufficiently small $\delta$, depending only on $T$, $\la_0$, $m$ and $p$, we can improve the bootstrap assumption \eqref{eq:cond4:A}. 


\bigskip{
	\footnotesize%
	
	\addvspace{\medskipamount}
	\textsc{School of Mathematics and Statistics, Wuhan University, Wuhan, China} \par
	\textit{E-mail address}: \texttt{shuang.m@whu.edu.cn}
	
	\addvspace{\medskipamount}
	\textsc{Beijing International Center for Mathematical Research, Peking University, Beijing, China} \par
	\textit{E-mail address}: \texttt{shiwuyang@math.pku.edu.cn}
	
	\addvspace{\medskipamount}
	\textsc{Department of Mathematical Sciences, Tsinghua University, Beijing, China} \par
	\textit{E-mail address}: \texttt{yupin@mail.tsinghua.edu.cn}
	}

\end{document}